\documentclass[12pt]{article}

\usepackage{amsmath,amssymb,amsthm}
\usepackage{hyperref}
\usepackage[utf8]{inputenc}
\usepackage{mathrsfs}
\usepackage[left=3cm,bottom=3cm]{geometry}
\usepackage{lineno}

\theoremstyle{plain}
\newtheorem{theorem}{Theorem}
\newtheorem{lemma}{Lemma}
\newtheorem{problem}{Problem}
\newtheorem{corollary}{Corollary}
\newtheorem{proposition}{Proposition}

\usepackage{xcolor}

\theoremstyle{definition}
\newtheorem{definition}{Definition}
\newtheorem{example}{Example}

\renewcommand{\Re}{\mathop{\mathrm{Re}}}

\DeclareMathOperator{\Arg}{Arg}
\DeclareMathOperator{\id}{id}

\numberwithin{equation}{section}
\numberwithin{example}{section}
\numberwithin{theorem}{section}
\numberwithin{proposition}{section}
\numberwithin{lemma}{section}
\numberwithin{definition}{section}
\numberwithin{problem}{section}

\begin{document}

\hfill {\it Dedicated to G. M. Henkin}\medskip
\begin{center}
    \Huge \bf Inverse problems in models\\ of resource distribution
\end{center}

\begin{center} \textit{A. D. Agaltsov}\footnote{(i) CMAP, Ecole Polytechnique, CNRS, Universit\'e Paris-Saclay, 91128 Palaiseau, France; (ii) Moscow Institute of Physics and Technology, Institutskiy Per. 9, Dolgoprudnyi 141700, Russia; email: alexey.agaltsov@polytechnique.edu}, \textit{E. G. Molchanov}\footnote{Moscow Institute of Physics and Technology, Institutskiy Per 9, Dolgoprudnyi 141700, Russia; email:molchanov.eg@mipt.ru}, \textit{A. A. Shananin}\footnote{(i) Moscow Institute of Physics and Technology, Institutskiy Per. 9, Dolgoprudnyi 141700, Russia; (ii) Russian Acad Sci., Fed. Res. Ctr. Comp. Sci. and Control, Ul. Vavilova 44-2, Moscow 119333, Russia; (iii) Moscow MV Lomonosov State University, Leninskiye Gory 1-52, Moscow 119991, Russia; (iv) Peoples Friendship Univ. Russia, Miklukho Maklaya Str. 6, Moscow 117198, Russia; email:shananin.aa@mipt.ru}\\ \bigskip

\today \end{center}

\begin{quote} \textbf{Abstract.} We continue to study the problem of modeling of substitution of production factors motivated by the need for computable mathematical models of economics that could be used as a basis in applied developments. This problem has been studied for several decades, and several connections to complex analysis and geometry has been established. We describe several models of resource distribution and discuss the inverse problems for the generalized Radon transform arising is these models. We give a simple explicit range characterization for a particular case of the generalized Radon transform, and we apply it to show that the most popular production functions are compatible with these models. Besides, we give a necessary condition and a sufficient condition for solvability of the model identification problem in the form of an appropriate moment problem. These conditions are formulated in terms of rhombic tilings.\medskip

\end{quote}

\section{Introduction. New problems of mathematical economics in the context of globalization}\label{sec.int}
The idea of being able to analyse social and economic processes using mathematical modeling, as it happens with physical phenomenons, has existed for a long time, at least since the XVIII century. However, the corpus of mathematical models in economy was not formed until the middle of the XX century. Then a special research area was created to improve the quality of gathering and analysis of economic statistics. In the 70s and 80s of the XX century the countries with advanced market relations successfully used computable models of economic equilibrium as a basis in applied developments and in making major social and economic decisions. Unfortunately, later the universality of such models did not meet expectations and the quality of economic statistics decreased. Once again, the dominant role in decision-making and economic analysis was given to expert opinions as opposed to model calculations. So what happened? We believe that this was caused by changes in the world economy that took place during the last quarter of the XX century. Today, one of the crucial tasks of mathematical economics is to study these changes and adapt mathematical models and statistics accordingly.

Globalization has been the main trend in the world economy since the last quarter of the XX century. The home-made products in domestic markets of the developing countries were forced to compete with the imported goods of the same kind. This led to standartization of products and to significant augmentation of substitutability of goods. At the same time economic statistics (e.g. Laspeyres' index) or models used in applied research (e.g. Leontief's input-output model of inter-industry balance) were based on the empirical hypothesis of constant consumption structure of final goods and production factors. Clearly, this hypothesis was significantly violated and, as a result, the gathered and analysed statistics were not able to reflect the economic processes. In turn, this led to problems of identification of generally accepted models. For example, in spite of all the efforts, the Russian national statistical authorities have not managed to indentify the inter-industry balance model since 2003. In the present work we study possible modifications of production models taking into account substitutability of production factors. We also describe and study inverse problems of integral and discrete geometry arising in these modified models.

Below we briefly describe the contents of the present article.

In Section \ref{sec.hjm} we recall the Houthakker-Johansen model of optimal resource distribution for an industry with substitution of production factors at the macro-level which goes back to \cite{Houth1955,Joh1972}. This model is based on Leontief's fixed proportion hypothesis at the micro-level and on the notion of distribution of capacities over micro-level technologies. In Subsection \ref{hjm.dir} we recall the generalized Neyman-Pearson lemma which gives necessary and sufficient  conditions for optimal distributions of resources, and we interpret these conditions as market-type equilibrium mechanisms.

In Subsection \ref{hjm.mic-mac} we recall the connection between the micro- and macro-descriptions of the industry established in \cite{Shananin1997ena,Shananin1997enb}. Next, we discuss the inverse problems of integral geometry (more precisely, inverse problems for the Radon transform with incomplete data) related to study of this connection, and investigated in \cite{Henkin1990a,Henkin1991}. In particular, we recall the range characterization Theorem \ref{hjm.thmchar} based on the local range characterization conditions for the Laplace transform obtained in \cite{Henkin1990a}, and based on Bernstein's theorems on completely monotone functions and separate analyticity. 

In Section \ref{sec.agr} we consider a model of optimal resource distribution for a group of industries producing different products and interconnected by the mutual resource supply. The resources are distributied in a way that maximizes the utility function for a representative rational end consumer.

In Proposition \ref{agr.propsol} of Subsection \ref{agr.pro} we give conditions characterizing optimal distributions of resources. In a similar way with the Houthakker-Johansen model, economic interpretation of these conditions is that the optimal resource distribution is carried out by market-type equilibrium mechanisms. In Proposition \ref{agr.propextr} we formulate the extremal principle which facilitates determination of optimal resource distributions. In Proposition \ref{agr.propdual} we establish that the aggregate production and profit functions are related by a Legendre-Fenchel-Young type transform. 

In Subsection \ref{agr.uni} we investigate universality of the Houthakker-Johansen model: is it possible to describe the whole group of industries using the Houthakker-Johansen model of a single industry? We state a related inverse problem, and recall an example of \cite{Henkin1990a} of a group of two industries that does not admit an aggregate description using the Houthakker-Johansen model. 

In Subsection \ref{agr.cor} we recall a necessary and sufficient condition of \cite{Karzanov2005} for a group of two industries to have an aggregate description using the Houthakker-Johansen model, in a particular case of industries with finite number of technologies and producing complementary goods using the same primary resources. This condition amounts to existence of a so-called stable correspondence between micro-level technologies of these industries. We also give an example of two industries admitting an aggregate description using the Houthakker-Johansen model despite the competition of micro-level technologies for primary resources.

Proofs of the results of Section \ref{sec.agr} are given in Appendix \ref{ape.agr}.

In Section \ref{sec.ghj} we recall the generalized Houthakker-Johansen model of optimal resource distribution for an industry with substitution of production factors at the micro-level which goes back to \cite{Shananin1997ena,Shananin1997enb}. Micro-level substitution is a typical feature of globalization, serving as a mechanism stabilizing inter-industrial connections. Micro-level substitution arises, first of all, due to standartization of production factors. In Subsection \ref{ghj.dir} we recall conditions of \cite{Shananin1997ena} characterizing optimal distributions and we interpret them as market-type equilibrium mechanisms. 

In Subsection \ref{ghj.inv} we recall the connection of between the micro- and macro-descriptions of the industry in the framework of the generalized Houthakker-Johansen model, and we discuss the inverse problems of integral geometry (more precisely, inverse problems for the generalized Radon transform) related to study of this connection, and studied in \cite{Agaltsov2015d,Agaltsov2016c}. In Proposition \ref{ghj.propchar} we give a simple and explicit characterization result for a particular case of completely monotone Radon transforms over the level curves of CES-functions. Applying this result, we obtain micro-founded descriptions for the Cobb-Douglas and CES production functions in the generalized Houthakker-Johansen model.

Proofs of the results of Section \ref{sec.ghj} are given in Appendix \ref{ape.ghj}.

In Section \ref{sec.ide} we study the identification problem for the generalized Houthakker-Johansen model. We state the moment problem: given the times series of outputs and prices of resources and outputs, determine their compatibility with the micro-level technological structure described by a given unit cost function. We recall a necessary and sufficient condition for solvability of this problem of \cite{Shananin1999}. In Proposition \ref{ide.propnes} we give a simple necessary condition for solvability of this problem. In Proposition \ref{ide.propsuf} we show that this necessary condition is also sufficient if a certain condition of discrete convexity is fulfilled.

Proofs of the results of Section \ref{sec.ide} are given in Appendix \ref{ape.ide}.

In Section \ref{sec.ele} we consider a particular case of the moment problem of Section \ref{sec.ide} corresponding to fixed elasticity of substitution at the micro-level. In Subsection \ref{ele.est} we state the problem of estimation of micro-level elasticity of substitution from the times series of outputs and prices of resources and outputs, and we show that this problem is solvable in polynomial time with respect to the length of the time series.

In Subsection \ref{ele.sim} we make a change of variables which allows to simplify the initial elasticity estimation problem. Next, to each elasticity we associate a partition of $\mathbb R^2_+$ by straight lines. Estimation of elasticity reduces to analysis of these partitions.

In Subsection \ref{ele.fwo} we show how to associate to each of these partitions of $\mathbb R^2_+$ a formal word, whose letters are elementary transpositions. We also show that variations of elasticity correspond to applications of the so-called 2-braid and 3-braid moves (also known as Moore-Coxeter relations) to this formal word.

In Subsection \ref{ele.til} we associate to each of the above partitions of $\mathbb R^2_+$ a rhombic tiling, and to each permutation we assign a polygonal chain called snake. In Proposition \ref{ele.snake+} we give a sufficient condition for solutions of the elasticity estimation problem: if the snake corresponding to the order $\lambda$ of the times series of the outputs belongs to the rhombic tiling corresponding to some fixed elasticity, this elasticity is a solution of the elasticity estimation problem. In Proposition \ref{ele.snake-} we give a similar necessary condition: if the snake corresponding to the above order $\lambda$ does not belong to the closed region bounded by the rhombic tiling associated to some fixed elasticity, then this elasticity is not a solution of the elasticity estimation problem. The intermediate case, when this snake corresponding to order $\lambda$ belongs to the closed region bounded by the rhombic tiling associated to some fixed elasticity, but does not belong to the tiling itself, is treated by Proposition \ref{ide.propsol}. Finaly, Proposition \ref{ele.LZ} states that if two words of partitions corresponding to some elasticities are equal, when considered as permutations, to an appropriately defined order of prices, then they can be obtained one from another using finite numbers of 2- and 3-braid moves.

Proofs of the results of Section \ref{sec.ele} are given in Appendix \ref{ape.ele}.

\section{Houthakker-Johansen model of resource distribution with substitution at the macro-level}\label{sec.hjm}

\subsection{Resource distribution problem}\label{hjm.dir}

We consider the resource distribution model introduced in \cite{Houth1955,Joh1972}, see also \cite{Hildenbrand1981}. This model describes the substitution of production factors at the macro-level (i.e., at the level of the industry as a whole) assuming the validity of the empirical Leontief hypothesis at the micro-level (i.e. at the level of single production capacities).

This model is based on the hypothesis of separation of timescales. According to this hypothesis, transformations of production capacities take place in the ``slow timescale'' and are related to management of capital funds. On the other hand, operational management of existing capacities, including resource supply and loading of capacities, takes place in the ``fast timescale''.

In the framework of the Houthakker-Johansen model an industry produces a homogeneous output using $n$ types of current use production factors (CUPF for short). Creation of new capacity or conversion of an existing one is related to the choice of production technology. This choice is made in the ``slow timescale''.

It is supposed that micro-level techologies are Leontief-type: they do not admit substitution of inputs and each of them can be described by vector $x = (x_1,\dots,x_n)$ of input expenses per unit of output. In the ``fast timescale'' capacities are distributed over technologies and are described by a non-negative measure $\mu$ on the space of all possible technologies $\mathbb R^n_+$. For a given Borel subset $A \subseteq \mathbb R^n_+$ the value $\mu(A)$ is the total capacity of technologies contained in $A$. The total capacity of the industry, i.e. its maximal possible output, is equal to $\mu(\mathbb R^n_+)$. 

In order to fully load all the capacities, the industry requires
\begin{equation*}
  l_i^* = \int_{\mathbb R^n_+} x_i \, \mu(dx)
\end{equation*}
resources of type $i$. Given a vector of inputs $l = (l_1,\dots,l_n)$ for the industry such that $l_i < l_i^*$ for at least one $i \in \{1,\dots,n\}$, the problem is to distribute the inputs over technologies. To each technology $x$ we assign a number $u(x) \in [0,1]$, which describes the load coefficient for capacities corresponding to technology $x$. If $u(x) = 1$, the capacities corresponding to technology $x$ are used completely. If $u(x) = 0$, the capacities corresponding to technology $x$ are not used.

Consider the problem of optimal distribution of resources for maximization of the total output of the industry:
\begin{gather}
  \int_{\mathbb R^n_+} u(x) \, \mu(dx) \to \max_{u(x)},\label{hjm.prmax} \\
  \int_{\mathbb R^n_+} x_i u(x) \mu(dx) \leq l_i \quad (i=1,\dots,n) \label{hjm.prres}\\
  0 \leq u(x) \leq 1 \quad \text{$\mu$-almost everywhere}.\label{hjm.pru}
\end{gather}

\begin{lemma}[Generalized Neyman-Pearson lemma, see, e.g., \cite{Shan1984en}]\label{hjm.lemmaGNP} Let $\mu$ be a locally finite non-negative Borel measure in $\mathbb R^n_+$. The following statements are valid:
\begin{itemize}
 \item[(i)] If $l \geq 0$, then problem \eqref{hjm.prmax}--\eqref{hjm.pru} is solvable.
 \item[(ii)] If $u_0(x)$ is a solution to problem \eqref{hjm.prmax}--\eqref{hjm.pru}, then there exist Lagrange multipliers $p_0 \geq 0$, $p = (p_1,\dots,p_n) \geq 0$, $p_0+|p|>0$, such that
 \begin{gather}
	u_0(x) = \begin{cases}
				0, & \text{for $\mu$-almost all $x$ such that $p_0 < p \cdot x$},\\
				1, & \text{for $\mu$-almost all $x$ such that $p_0 > p \cdot x$},
	         \end{cases}\\
	p_i \biggl( l_i - \int_{\mathbb R^n_+} x_i u_0(x) \mu(dx) \biggr) = 0 \quad (i = 1,\dots,n).
 \end{gather}
 \item[(iii)] If
 \begin{equation*}
	p_0 > 0, \; p = (p_1,\dots,p_n) \geq 0, \; l = \int_{\mathbb R^n_+} x \theta(p_0 - p \cdot x) \mu(dx),
 \end{equation*}
 then $u(x) = \theta(p_0 - p \cdot x)$ is a solution of problem \eqref{hjm.prmax}--\eqref{hjm.pru}. Here and in what follows:
 \begin{equation}\label{hjm.thetadef}
	\theta(s) = \begin{cases}
				  1, & \text{if $s \geq 0$},\\
				  0, & \text{if $s < 0$}.
	            \end{cases}
 \end{equation}
\end{itemize}
\end{lemma}

Lagrange multipliers $p_0 \geq 0$, $p = (p_1,\dots,p_n) \geq 0$ are interpreted as the unit prices of the output and resources, respectively. The generalized Neyman-Pearson lemma states that the optimal resource distribution mechanisms are equivalent to equilibrium market-type mechanisms: technologies which are profitable for prices $p_0 \geq 0$, $p = (p_1,\dots,p_n) \geq 0$ are used completely, unprofitable technologies are not used. Furthermore, prices $p_0 \geq 0$, $p = (p_1,\dots,p_n) \geq 0$ are determined by the condition of equilibrium of supply and demand on the market of CUPF.

The described model can be fruitfully used in analysis of real economic problems, see, e.g., \cite{Joh1972,Hildenbrand1981}. In particular, this model was used in investigation of the crisis of the Norwegian tanker fleet in 1960-ies, see \cite{Joh1972}.

\begin{definition}\label{hjm.Fdef} Function $F(l)$ which assigns to each $l = (l_1,\dots,l_n) \geq 0$ the optimal value of functional in the resource distribution problem \eqref{hjm.prmax}--\eqref{hjm.pru} is called \textit{production function} for this problem.
\end{definition}

\begin{definition}
Function $\Pi(p,p_0)$ given by
\begin{gather}
  \Pi(p,p_0) = \int_{\mathbb R^n_+} (p_0 - p \cdot x)_+ \mu(dx),  \label{hjm.Pidef}\\
  a_+ = \max\{a,0\}, \quad a \in \mathbb R, \label{hjm.a+}
\end{gather}
is called \textit{profit function} for the resource distribution problem \eqref{hjm.prmax}--\eqref{hjm.pru}.
\end{definition}

Production function $F(l)$ assigns to a given vector of inputs $l = (l_1,\dots,l_n)$ the the total output of the industry. It describes substitution of production factors at the macro-level. $F(l)$ has neoclassical properties: it is concave, non-decreasing and continuous on $\mathbb R^n_+$, see \cite{Shan1984en}.

The definition of profit function \eqref{hjm.Pidef} goes back to \cite{Cornwall1973,Shan1985en}. Properties of the profit function are studied in \cite{Shan1985en}. In particular, it is shown that the profit function is related to the production function by the following Legendre-Young-Fenchel type transform:
\begin{gather}
  \Pi(p,p_0) = \sup_{l \geq 0} ( p_0 F(l) - p \cdot l), \quad p_0 > 0, \; p = (p_1,\dots,p_n) \geq 0, \label{hjm.FtoPi}\\
  F(l) = \tfrac{1}{p_0} \inf_{p \geq 0} ( \Pi(p,p_0) + p \cdot l), \quad p_0>0, \; l=(l_1,\dots,l_n)\geq 0. \label{hjm.PitoF}
\end{gather}
Thus, the production function and the profit function contain the same macro-level information.  

\subsection{Connection between micro- and macro-descriptions of an industry and related inverse problems}\label{hjm.mic-mac}

In \cite{Houth1955} it is shown that the Cobb-Douglas production function
\begin{equation}\label{hjm.FCDdef}
  F_\text{CD}(l_1,l_2) = C l_1^{\frac{\alpha_1}{\alpha_1+\alpha_2+1}} l_2^{\frac{\alpha_2}{\alpha_1+\alpha_2+1}}, \quad C > 0, \; \alpha_1 \geq 1, \; \alpha_2 \geq 1,
\end{equation}
corresponds to distribution $\mu$ of capacities over technologies which is given by:
\begin{equation}\label{hjm.FCDconst}
 \begin{gathered}
	  \mu(dx) = A x_1^{\alpha_1-1} x_2^{\alpha_2-1} dx, \\
	  A = C^{\alpha_1+\alpha_2+1}  \tfrac{(\alpha_1+\alpha_2)\alpha_1^{\alpha_1} \alpha_2^{\alpha_2}}{(\alpha_1+\alpha_2+1)^{\alpha_1+\alpha_2}} \tfrac{1}{B(\alpha_1,\alpha_2)},
 \end{gathered}
\end{equation}
where $B(\cdot,\cdot)$ is the beta function. This example shows that even if there are no substitution at the micro-level, substitution can appear at the macro-level due to differences in loads of capacities corresponding to different micro-level technologies. This example also encourages one to take a critical look at the Cobb-Douglas production function which is very popular in econometric studies: the corresponding distribution of capacities has an asymptotic ``horn of plenty'' since it does not vanish in a neighborhood of zero.

This example motivates the study of connection between distributions $\mu(dx)$ of capacities over technologies and corresponding production functions $F(l)$. In view of \eqref{hjm.FtoPi}, \eqref{hjm.PitoF}, study of this relation is equivalent to investigation of operator \eqref{hjm.Pidef} and transforms \eqref{hjm.FtoPi}, \eqref{hjm.PitoF}.

Integral operator \eqref{hjm.Pidef} is related to Radon transform in $\mathbb R^n_+$:
\begin{equation}\label{hjm.PitoRad}
  \frac{\partial^2 \Pi(p,p_0)}{\partial p_0^2} = \int_{p \cdot x = p_0} \mu(dx), \quad p_0 > 0, \; p = (p_1,\dots,p_n)\geq 0,
\end{equation}
and, as a corollary, to classical Fourier-Laplace and Cauchy-Fantappi\`e transforms, see \cite{Shan1985en,Henkin1990a,Henkin1990b,Henkin1991}. In particular,
\begin{equation}\label{hjm.RtoL}
  \int_{\mathbb R^n_+} e^{-s \cdot x} \mu(dx) = \int_0^{+\infty} e^{-\tau} d_\tau \bigl( \tfrac{\partial \Pi(s,\tau)}{\partial \tau} \bigr), \quad s = (s_1,\dots,s_n) \geq 0.
\end{equation}
Relation \eqref{hjm.RtoL} allows to study the questions of inversion and range characterization for operator \eqref{hjm.Pidef}.

\begin{theorem}[see \cite{Shan1985en,Henkin1990a}]\label{hjm.thmuni} Let $\mu$ be a signed Borel measure in $\mathbb R^n_+$ satisfying
\begin{equation}\label{hjm.muest}
  \int_{\mathbb R^n_+} e^{-A|x|} |\mu(dx)|< \infty \quad \text{for some $A > 0$}.
\end{equation}
Suppose that there exists an open non-empty cone $K \subseteq \mathbb R^n_+$ such that
\begin{equation}
  \int_{\mathbb R^n_+} (p_0 - p\cdot x)_+ \mu(dx) = 0 \quad \text{for all $p_0 > 0$, $p \in K$}.
\end{equation}
Then $\mu = 0$. 
\end{theorem}

\begin{theorem}[see \cite{Henkin1990a}]\label{hjm.thmchar} Function $\Pi(p,p_0)$ can be represented in the form
\begin{equation}
  \Pi(p,p_0) = \int_{\mathbb R^n_+} (p_0 - p \cdot x)_+ \mu(dx), \quad (p,p_0) \in \mathbb R^{n+1}_+,
\end{equation}
where $\mu$ is a non-negative Borel measure in $\mathbb R^n_+$ satisfying \eqref{hjm.muest}, if and only if
\begin{enumerate}
 \item[(i)] $\Pi(p,p_0) = p_0 \Pi(\tfrac{p}{p_0},1)$ for $(p,p_0) \in \mathbb R^{n+1}_+$, $\Pi(p,p_0)$ is convex in $\mathbb R^{n+1}_+$, and for fixed $p \in \mathbb R^n_+$ the measure $\tfrac{\partial^2 \Pi(p,\tau)}{\partial \tau^2}$ decays exponentially as $\tau \to + \infty$.
 \item[(ii)] Put
 \begin{equation}
	G(s) = \int_0^{+\infty} e^{-\tau} d_\tau \bigl( \tfrac{\partial \Pi(s,\tau)}{\partial \tau} \bigr).
 \end{equation}
Then $G(s) \in C^\infty(\mathbb R^n_+)$ and the following inequalities hold:
  \begin{equation}\label{hjm.Lapchar}
   \begin{gathered}
	  (-1)^k D_{\xi^1} \cdots D_{\xi^k} G(\lambda s) \geq 0\\
	  \text{for some open cone $\Gamma \in \mathop{\mathrm{int}} \mathbb R^n_+$ and some $s \in \Gamma$},\\
	  \text{for any $\lambda > 0$, any $\xi^1$, \dots, $\xi^k \in \Gamma$ and any $k \geq 1$},\\
	  \text{where $D_\xi = \sum\nolimits_j \xi_j \tfrac{\partial}{\partial s_j}$, $\xi = (\xi_1,\dots,\xi_n)$}.
   \end{gathered}
  \end{equation}
\end{enumerate}
\end{theorem}

The proof of the range characterization theorem \ref{hjm.thmchar} is based on generalizations of Bernstein's theorems on completely monotone functions and separate analyticity.

Theorem \ref{hjm.thmchar} allows to determine which substitution of resources at the macro-level is compatible with the Houthakker-Johansen model.

\begin{example}[See \cite{Henkin1990a} for details] Consider the following constant elasticity of substitution production function (CES for short):
\begin{equation}\label{hjm.FCES}
  \begin{gathered}
  F_\text{CES}(l_1,l_2) = (\alpha_1 l_1^{-\rho} + \alpha_2 l_2^{-\rho})^{-\frac \gamma \rho}, \\
  \alpha_1 > 0, \; \alpha_2 > 0, \; \rho \in [-1,0)\cup(0,+\infty), \; 0 < \gamma < 1.
  \end{gathered}
\end{equation}
The corresponding profit function, calculated using formula \eqref{hjm.FtoPi}, is given by
\begin{equation}\label{hjm.PiCES}
  \Pi_\text{CES}(p_1,p_2,p_0) = \gamma^{\frac{\gamma}{1-\gamma}} (1-\gamma) p_0^{\frac{1}{1-\gamma}} \bigl( \alpha_1^{\frac{1}{1+\rho}} p_1^{\frac{\rho}{1+\rho}} + \alpha_2^{\frac{1}{1+\rho}} p_2^{\frac{\rho}{1+\rho}} \bigr)^{\frac{-\gamma(1+\rho)}{\rho(1-\gamma)}}.
\end{equation}
It is shown in \cite{Henkin1990a} that the range characterization conditions of Theorem \ref{hjm.thmchar} are fulfilled for function $\Pi_\text{CES}$ of \eqref{hjm.PiCES} for $\rho \in (-1,0)\cup(0,+\infty)$ and are not fulfilled for $\rho=-1$. Thus, CES production functions \eqref{hjm.FCES} are compatible with the Houthakker-Johansen model for $\rho \in (-1,0)\cup(0,+\infty)$. Also note that it was shown in \cite{Shan1984en} that function \eqref{hjm.FCES} is not compatible with the Houthakker-Johansen model for $\gamma = 1$.
\end{example}

Besides, Theorem \ref{hjm.thmchar} allows to investigate universality of the Houthakker-Johansen model from the point of view of aggregation, see Subsection \ref{agr.uni}.

\section{Aggregation and the inverse problem}\label{sec.agr}
\subsection{Resource distribution problem for a group of industries and aggregation}\label{agr.pro}
Consider a group of $m$ industries, where the $i$-th industry produces a homogeneous product of $i$-th type. These industries are interconnected by the mutual resource supply: the output of the $i$-th industry is used by other industries as a production factor. We set:
\begin{equation}
\begin{gathered}
  \text{$X^j_i$: amount of the $i$-th product used by the $j$-th industry},\\
  X^j = (X_1^j,\dots,X_m^j),\\
  \text{$l^j  = (l_1^j,\dots,l_n^j)$: vector of primary resources used by the $j$-th industry.}
\end{gathered}
\end{equation}
Let $F_j(X^j,l^j)$ be the production function for the $j$-th industry, that is, the function that relates the inputs to the output. This function is supposed to be neoclassical: concave, non-decreasing in each variable, continuous in $\mathbb R^n_+$ and vanishing at the origin.

We also put:
\begin{equation}
  \begin{gathered}
  \text{$X^0_i$: amount of the $i$-th product delivered to end consumers},\\
  X^0 = (X_1^0,\dots,X_m^0).
  \end{gathered}
\end{equation}
The demand of end consumers is described by utility function $F_0(X^0)$. It is supposed that $F_0(X^0)$ is positive homogeneous of degree one, concave, continuous in $\mathbb R^m_+$ and positive in $\mathop{\mathrm{int}}\mathbb R^n_+$. For the economic content of these assumptions, see \cite{Shananin2009}.

The primary resources are supposed to be limited by volumes $l = (l_1,\dots,l_n) \geq 0$. Consider the following resource distribution problem for maximization of utility of the end consumers, taking into account the inter-industrial balance of resources:
\begin{gather}
  F_0(X^0) \to \max, \label{agr.prmax}\\
  F_j(X^j,l^j) \geq X^0_j + X^1_j + \cdots + X^m_j \quad (j=1,\dots,m), \label{agr.prbal}\\
  l^1 + \cdots + l^m \leq l, \label{agr.prres}\\
  \begin{gathered}\label{agr.prpos}
  X^0 \geq 0, \; X^1 \geq 0, \; \dots, \; X^m \geq 0, \\
  l^1 \geq 0, \; \dots, \; l^m \geq 0.
  \end{gathered}
\end{gather}
It is supposed that the considered group of industries is \textit{productive}: 
\begin{equation}\label{agr.prod}
  \begin{gathered}
 \text{there exist $\widehat X^1 \geq 0$, \dots, $\widehat X^m \geq 0$, $\widehat l^1 \geq 0$, \dots, $\widehat l^m \geq 0$}\\
 \text{such that $F_j(\widehat X^j,\widehat l^j) > \widehat X^0_j + \cdots + \widehat X^m_j$, $j = 1,\dots,m$}.
 \end{gathered}
\end{equation}

\begin{proposition}\label{agr.propsol} A set of vectors $\{\widehat X^0,\widehat X^1,\dots,\widehat X^m, \widehat l^1,\dots,\widehat l^m\}$ satisfying \eqref{agr.prbal}--\eqref{agr.prod} is a solution of optimization problem \eqref{agr.prmax}--\eqref{agr.prpos} if and only if there exist Lagrange multipliers $p_0>0$, $q = (q_1,\dots,q_m)\geq 0$, $s = (s_1,\dots,s_n) \geq 0$, such that
\begin{gather}
  (\widehat X^j,\widehat l^j) \in \Arg \max \bigl\{ q_j F_j(X^j,l^j) - q \cdot X^j - s \cdot l^j \mid X^j \geq 0, \; l^j \geq 0 \bigr\}, \;\; j \geq 1, \label{agr.argmaxcond}\\
  q_j \bigl( F_j(\widehat X^j,\widehat l^j) - \widehat X^0_j - \widehat X^1_j - \cdots - \widehat X^m_j \bigr) = 0 \quad (j=1,\dots,m), \label{agr.qequ} \\
  s_k \bigl( l_k - \widehat l^1_k - \cdots - \widehat l^m_k \bigr) = 0 \quad ( k = 1,\dots,n), \label{agr.sequ}\\
  \widehat X^0 \in \Arg\max \bigl\{ p_0 F_0(X^0) - q \cdot X^0 \mid X^0 \geq 0 \bigr\}. \label{agr.argmaxcond2}
\end{gather}
\end{proposition}
The proof of Proposition \ref{agr.propsol} is given in Appendix \ref{ape.agr}.

Lagrange multipliers $q = (q_1,\dots,q_m)$ corresponding to balance constraints on products of industries are interpreted as prices of these products. Lagrange multipliers $s = (s_1,\dots,s_n)$ corresponding to balance constraints on primary resources are interpreted as prices of these resources. 

The \textit{profit function} of the $j$-th industry is defined as
\begin{equation}
  \Pi_j(q,s) = \sup \bigl\{ q_j F_j(X^j,l^j) - q \cdot X^j - s \cdot l^j \mid X^j \geq 0, \; l^j \geq 0 \bigr\}.
\end{equation}
Formula \eqref{agr.argmaxcond} means that the product supply and resource demand of the industries are determined by the condition of profit maximization for prices $(q,s)$.

Formula \eqref{agr.argmaxcond2} describes the demand of a representative rational end consumer with utility function $F_0(X^0)$ for prices $q$ of products. Furthermore, we have that $p_0 = q_0(q)$, where $q_0(q)$ is the Young transform of function $F_0(X^0)$:
\begin{equation}
  q_0(q) = \inf \bigl\{ \tfrac{q \cdot X^0}{F_0(X^0)} \mid X^0 \geq 0, \;F_0(X^0) >0 \bigr\},
\end{equation}
see \cite{Shananin2009} for more details.

It follows from formulas \eqref{agr.prbal}, \eqref{agr.prres}, \eqref{agr.qequ}, \eqref{agr.sequ} that $q = (q_1,\dots,q_m)$ and $s = (s_1,\dots,s_n)$ are equilibrium prices. Thus, the optimal resource distribution mechanisms are equivalent to market-type equilibrium mechanisms. 

\begin{proposition}[Extremal principle]\label{agr.propextr} Lagrange multipliers $(q,s)$ for optimization problem \eqref{agr.prmax}--\eqref{agr.prpos} are solutions of the following convex programming problem:
\begin{equation}\label{agr.prPimax}
  \begin{gathered}
	\Pi_1(q,s)+\cdots+\Pi_m(q,s) \to \min_{q,s}, \\
	q_0(q) \geq p_0, \quad q \geq 0.
	\end{gathered}
\end{equation}
\end{proposition}
Proposition \ref{agr.propextr} follows from comparison of conditions \eqref{agr.prbal}--\eqref{agr.prpos}, \eqref{agr.argmaxcond}--\eqref{agr.argmaxcond2} with necessary and sufficient optimality conditions for problem \eqref{agr.prPimax}, obtained using the Kuhn-Tucker theorem, see \cite{Shananin1997enb,Shananin1999b} for more details.

\begin{definition} Function $F^A(l)$ which assigns to each $l=(l_1,\dots,l_n)\geq 0$ the optimal value of functional in optimization problem \eqref{agr.prmax}--\eqref{agr.prpos} is called the \textit{aggregate production function} for the resource distribution problem \eqref{agr.prmax}--\eqref{agr.prpos}.
\end{definition}

\begin{definition} Function $\Pi^A(s,p_0)$ given by
\begin{equation}\label{agr.Pidef}
  \Pi^A(s,p_0) = \min \bigl\{ \Pi_1(q,s)+\cdots+\Pi_m(q,s) \mid q_0(q) \geq p_0, \; q \geq 0 \bigr\}
\end{equation}
is called the \textit{aggregate profit function} for the resource distribution problem \eqref{agr.prmax}--\eqref{agr.prpos}.
\end{definition}

In view of Proposition \ref{agr.propextr}, function $\Pi^A(s,p_0)$ assigns to each vector of unit prices of primary resources $s = (s_1,\dots,s_n) \geq 0$ and to \textit{aggregate product's consumer price index} $p_0>0$ the total profit of the group of industries in conditions of equilibrium.

\begin{proposition}\label{agr.propdual}
  The following formulas hold:
  \begin{align}
	F^A(l) & = \tfrac{1}{p_0} \inf_{s \geq 0} \bigl( \Pi^A(s,p_0) + s \cdot l \bigr), \quad p_0>0, \; l = (l_1,\dots,l_n)\geq 0, \label{agr.PitoF}\\
	\Pi^A(s,p_0) & = \sup_{l \geq 0} \bigl( p_0 F^A(l) - s \cdot l \bigr), \quad p_0 > 0, \; s = (s_1,\dots,s_n)\geq 0. \label{agr.FtoPi}
  \end{align}
\end{proposition}
Proposition \ref{agr.propdual} is proved in Appendix \ref{ape.agr}.

\subsection{Is the Houthakker-Johansen model the universal resource distribution model?}\label{agr.uni}

We consider the question of \textit{universality} for the Houthakker-Johansen model: is it possible to describe the macro-level behavior of a group of industries using the Houthakker-Johansen model of a single industry? In view of \eqref{agr.PitoF}, \eqref{agr.FtoPi} this question leads to the following inverse problem.
\begin{problem} Given function $\Pi^A(s,p_0)$ defined in \eqref{agr.Pidef}, find a non-negative measure $\mu^A$ supported in $\mathbb R^n_+$ such that
\begin{equation}\label{agr.Pimu}
  \Pi^A(s,p_0) = \int_{\mathbb R^n_+} ( p_0 - s \cdot x)_+ \mu^A(dx),
\end{equation}
where $(\cdot)_+$ is defined in \eqref{hjm.a+}.
\end{problem}

If $n=1$, equation \eqref{agr.Pimu} is uniquely solvable and
\begin{equation}
  \mu^A(dx) = d_x \bigl( \tfrac{\partial \Pi^A(1,x)}{\partial x} \bigr).
\end{equation}
Thus, in the case of a single primary resource the aggregate distribution of capacities over technologies is well-defined. Note that the notion of the aggregate distribution of capacities over technologies is useful for substantive analysis of economic systems (see, e.g., \cite{Petrov1996}), and it is similar to the notion of quasiparticle in physics.

It follows from \eqref{agr.Pidef} that $\Pi^A(s,p_0)$ is convex and positive homogeneous of degree one. Suppose that each of the industries of the group are described by the Houthakker-Johansen model with compactly supported distributions of capacities. Then $\Pi^A(s,p_0)$ also satisfies conditions of the uniqueness Theorem \ref{hjm.thmuni} and condition (i) of the range characterization Theorem \ref{hjm.thmchar}. The following example shows that condition (ii) of the range characterization Theorem \ref{hjm.thmchar} can be violated.

\begin{example}[See \cite{Henkin1990a}]\label{agr.excomp} Consider the group of two industries using two types of primary resources (i.e. $m = 2$, $n=2$) and independently producing complementary products, so that
\begin{equation}\label{agr.F0comp}
  F_0(X_1^0,X_2^0) = \min(X_1^0,X_2^0).
\end{equation}
The industries are described by the Houthakker-Johansen model as follows.

There is a single technology in the first industry and the corresponding distribution of capacities is
\begin{equation}\label{agr.exmu1}
  \mu_1(dx) = k_0 \delta(x-z), \quad z = (z_1,z_2),
\end{equation}
where $\delta$ denotes Dirac's delta function. 

There are two technologies in the second industry and the corresponding distribution of capacities is 
\begin{equation}\label{agr.exmu2}
  \mu_2(dx) = k_1 \delta(x-y^1) + k_2 \delta(x-y^2), \quad y^j = (y^j_1,y^j_2), \quad j = 1,2.
\end{equation}

Is is supposed that
\begin{equation}
  k_1 + k_2 > k_0, \quad y_1^1 > y_1^2, \quad y_2^1 > y_2^2.
\end{equation}

Using \eqref{agr.Pidef}, we obtain
\begin{equation}\label{agr.PiAmax}
\begin{gathered}
	\Pi^A(s,p_0) = \max \bigl\{ \Pi_1(s,p_0), \Pi_2(s,p_0) \bigr\}, \\
	\Pi_1 = (k_0 - k_2)_+ \bigl( p_0 - s \cdot (z+y^1) \bigr)_+ + \min(k_0,k_2) \bigl(p_0 - s \cdot (z+y^2) \bigr)_+, \\
	\Pi_2 =  \min(k_0,k_1) \bigl( p_0 - s \cdot (z+y^1) \bigr)_+ + (k_0-k_1)_+ \bigl(p_0 - s \cdot (z+y^2) \bigr)_+,
 \end{gathered}
 \end{equation}
where $(\cdot)_+$ is defined in \eqref{hjm.a+}. Formula \eqref{agr.PiAmax} can be rewritten as follows:
\begin{equation}
  \begin{gathered}
  \Pi^A(s,p_0) = \Pi_j(s,p_0), \quad p_0 >0, \; s \in K_j, \; j = 1,2,\\ 
   K_1 = \bigl\{ s \in \mathbb R^2_+ \mid s\cdot y^2 \leq s \cdot y^1 \bigr\},\\
   K_2 = \bigl\{ s \in \mathbb R^2_+ \mid s \cdot y^1 \leq s \cdot y^2 \bigr\}.
   \end{gathered}
\end{equation}

Profit functions $\Pi_1(s,p_0)$, $\Pi_2(s,p_0)$ correspond to the following aggregate distributions of capacities over technologies:
\begin{equation}
  \begin{gathered}
   \mu_1^A(dx) = (k_0-k_2)_+ \delta(x-z-y^1) + \min(k_0,k_2) \delta(x-z-y^2), \\
   \mu_2^A(dx) = \min(k_0,k_1) \delta(x-z-y^1) + (k_0-k_1)_+ \delta(x-z-y^2),
   \end{gathered}
\end{equation}
respectively. If $s \in K_1$ (resp. $s \in K_2$) then the group of industries is described by a single aggregate distribution of capacities $\mu_1$ (resp. $\mu_2$). It follows from the uniqueness Theorem \ref{hjm.thmuni} that there are no unique aggregate distribution of capacities that could describe this group of industries for all $s \in \mathbb R^2_+$.

When prices $s$ of primary resources pass from cone $K_1$ to cone $K_2$, the relative profitability of technologies of the second industry is changed. This leads to change of the aggregate distribution of capacities. Probably, this effect explains the rapid decrease in power consumption for industries in developed capitalist countries after the energetic crisis of 1970ies.

Formally, this effect is caused by violation of condition (ii) of the range characterization Theorem \ref{hjm.thmchar}: function
\begin{equation}
  G(s) = \max\{ G_1(s), G_2(s) \}, \quad G_j(s) = \int_0^{+\infty} e^{-\tau} d_\tau \bigl( \tfrac{\partial \Pi_j(s,\tau)}{\partial \tau} \bigr) \quad (j = 1,2),
\end{equation}
is not smooth on the hypersurface $\bigl\{ s \in \mathbb R^2_+ \mid s \cdot y^1 = s \cdot y^2 \bigr\}$.
\end{example}

\subsection{Existence of an aggregate description and stable correspondences}\label{agr.cor}

We consider a generalization of Example \ref{agr.excomp} to the case of a group of two industries, each of which has a finite number of micro-level technologies. It is supposed that they use the same primary resources and produce complementary goods, so that the consumer's demand is described by utility function $F_0(X_1^0,X_2^0)$ of \eqref{agr.F0comp}. An example of such a group of industries is a car manufacturing company, where the first industry produces car engines and the second one produces the remaining parts of cars. 

Let $X = \{x^1,\dots,x^m\} \subset \mathbb R^n_+$, $Y = \{ y^1,\dots,y^m \} \subset \mathbb R^n_+$ be some multisets (i.e. the same element can occur multiple times) of unit capacity technologies of the first and second industry, respectively. Production capacities of these industries are distributed over technologies. Technologies differ by the amount of primary resource used per unit of the output.

In a similar way with Example \ref{agr.excomp}, for any fixed prices $s$ of raw resources, some capacities of the industries are used and some are not. The used capacities are determined by relative profitability of their technologies for prices $s$. Thus, for any fixed $s$ there is a correspondence $\gamma_s \in X \times Y$ such that $(x^i,y^j) \in \gamma_s$ if and only if $x^i$ and $y^j$ are used or not used at the same time. If the correspondence $\gamma_s$ changes as we change $s$, this is interpreted as a transformation of technological structure of the group of industries. This transformation usually appears as an economic crisis.

If the correspondence $\gamma_s$ does not change for $s$ in some cone $K \subset \mathbb R^n_+$, then this pair of industries admits an aggregate description in the Houthakker-Johansen model for $s \in K$. In the following proposition we give a criterion for stability of such correspondences.

\begin{definition}\label{agr.defKst} Let $K \subseteq \mathbb R^n_+$ be a cone, $X = \{x^1,\dots,x^m\}$, $Y = \{y^1,\dots,y^m\}$ be multisets. Bijection $\gamma \colon X \to Y$ is called \textit{$K$-stable correspondence} if for any $x^i$, $x^j \in X$, $s \in K$, the inequality $s \cdot x^i < s \cdot x^j$ implies the inequality $s \cdot \gamma(x^i) \leq s \cdot \gamma(x^j)$. 
\end{definition}

If $\gamma \colon X \to Y$ is a $K$-stable correspondence then for any $s \in K$, capacities corresponding to technology $x^i \in X$ are used if and only if capacities corresponding to technology $\gamma(x^i) \in Y$ are used.   

For a given cone $K \subseteq \mathbb R^n_+$, we define its dual cone $K^*$ is follows:
\begin{equation}\label{agr.K*def}
  K^* = \bigl\{ p \in \mathbb R^n \mid \text{$p \cdot s \geq 0$ for any $s \in K$} \bigr\}. 
\end{equation}
Not that $x^j - x^i \in K^*$ if and only if $x^i \cdot s \leq x^j \cdot s$ for \textit{any} $s \in K$.

\begin{proposition}[see \cite{Karzanov2005}] Let $X$, $Y$, $K$ be the same as in Definition \ref{agr.defKst}. Bijection $\gamma \colon X \to Y$ is $K$-stable if and only if for any $x^i$, $x^j \in X$:
\begin{enumerate}
 \item[(i)] If $x^i \neq x^j$, $x^j - x^i \in K^*$, then $\gamma(x^j)-\gamma(x^i) \in K^*$.
 \item[(ii)] If $x^j - x^i \not\in K^*$, $x^i - x^j \not\in K^*$, then there exist $\lambda \geq 0$, $\mu \geq 0$, $\lambda + \mu > 0$, such that $\lambda(x^j-x^i) = \mu\bigl( \gamma(x^j) - \gamma(x^i) \bigr)$. 
\end{enumerate}
\end{proposition}
Note that (i) states that $\gamma$ is monotone with respect to order $\preceq$ defined as follows:
\begin{equation}
 \text{$b \preceq a$ if and only if $a - b \in K^*$, $a$, $b \in \mathbb R^n_+$}.
\end{equation}
Also note that condition (ii) is very restrictive: if there are technologies in the first (or second) industry which compete for resources for $s \in K$, distributions of capacities of the two industries must be similar.

Next, we show that even if there is a competition of micro-level technologies for primary resources, it is possible that two industries have an aggregate description.

\begin{example} We consider a group of two industries using two types of primary resources and independently producing partially substitutable goods. Consumer's demand is described by the following utility funtion:
\begin{equation}
  F_0(X_1^0,X_2^0) = \bigl( (X_1^0)^{-\rho} + (X_2^0)^{-\rho} \bigr)^{-\frac 1 \rho}, \quad \rho \in [-1,0) \cup (0,+\infty).
\end{equation}
The industries are described by the Houthakker-Johansen model with distributions of capacities over technologies given by \eqref{agr.exmu1}, \eqref{agr.exmu2}, respectively.

The aggregate profit function is given by
\begin{equation}
  \begin{gathered}
  \Pi^A(s,p_0) = \min\bigl\{ k_0 (p_1-s\cdot z)_+ + k_1(p_1 - s \cdot y^1)_+ \\
  + k_2(p_2 - s \cdot y^2)_+ 
  \bigm| \bigl( p_1^{\frac{\rho}{1+\rho}} + p_2^{\frac{\rho}{1+\rho}}\bigr)^{\frac{1+\rho}{\rho}} \geq p_0 \bigr\}.
  \end{gathered}
\end{equation}
In particular, if
\begin{equation}\label{agr.exp0diap}
  \begin{gathered}
  p_0 > \max \bigl\{  \kappa_1 s \cdot z,  \kappa_2 \max\{s\cdot y^1,s\cdot y^2\} \bigr\}, \\
	\kappa_1 = \biggl(\frac{k_0^\rho + (k_1+k_2)^\rho}{(k_1+k_2)^\rho} \biggr)^{\frac{1+\rho}{\rho}}, \quad \kappa_2 = \biggl(\frac{k_0^\rho + (k_1+k_2)^\rho}{k_0^\rho} \biggr)^{\frac{1+\rho}{\rho}},
  \end{gathered}
\end{equation}
then 
\begin{equation}
  \begin{gathered}
	\Pi^A(s,p_0) = \tfrac{k_0}{\kappa_1} \bigl( p_0 - \kappa_1 s\cdot z\bigr)_+ + \tfrac{k_1}{\kappa_2} \bigl(p_0 - \kappa_2 s \cdot y^1\bigr)_+ + \tfrac{k_2}{\kappa_2} \bigl(p_0-\kappa_2 s \cdot y^2\bigr)_+.
  \end{gathered}
\end{equation}
This aggregate profit function corresponds to the following distribution of capacities:
\begin{equation}
  \mu^A(dx) = \tfrac{k_0}{\kappa_1} \delta( x - \kappa_1 z) + \tfrac{k_1}{\kappa_2} \delta(x - \kappa_2 y^1) + \tfrac{k_2}{\kappa_2} \delta(x-\kappa_2 y^2 ),
\end{equation}
where $\delta$ is Dirac's delta function.

Thus, for prices $p_0$, $s$ satifsying \eqref{agr.exp0diap} the considered group of two industries admits a description using an aggregate distribution of capacities, despite the competition of micro-level technologies for primary resources. 
\end{example}

\section{Generalized Houthakker-Johansen model with substitution at the micro-level and new problems of integral geometry}\label{sec.ghj}

\subsection{Resource distribution problem taking into account micro-level substitution of inputs}\label{ghj.dir}

Competition of resources (for example, competition of home-made and imported resources) is a typical feature of globalization. Globalization of the world economy led to standartization of products produced in different countries. As a result, the micro-level substitution of production factors has become a mechanism that stabilizes inter-industrial connections. In the present section we describe a resource distribution model that takes into account these stabilizing mechanisms.

We use definitions and notations of Section \ref{sec.hjm}. Substitution of production factors at the micro-level is described by function $f(v_1,\dots,v_n)$ with neoclassical properties: it is positive homogeneous of degree one, concave and continuous on $\mathbb R^n_+$. It is supposed that a unit capacity with technology $x = (x_1,\dots,x_n)$ has production function 
\begin{equation}
  u \mapsto \min\bigl( 1, f(\tfrac{u_1}{x_1},\dots,\tfrac{u_n}{x_n})\bigr),
\end{equation}
where $u = (u_1,\dots,u_n)$ is the vector of inputs. 

In economic literature a typical example of such function $f(v_1,\dots,v_n)$ describing substitution of production factors is the \textit{constant elasticity of substitution} (CES) function 
\begin{equation}
   f(v_1,\dots,v_n) = (v_1^{-\rho} + \cdots + v_n^{-\rho} )^{-\frac 1 \rho}, \quad \rho \in [-1,0)\cup(0,+\infty).
\end{equation}
Parameter $\rho$ is related to elasticity $\sigma$ of substitution of production factors at the microlevel by $\sigma = \tfrac{1}{1+\rho}$. Note that:
\begin{enumerate}
 \item[(i)] if $\rho = -1$, production factors are perfect substitutes;
 \item[(ii)] if $\rho \in [-1,0)$, there is a competition of production factors with possible complete expulsion of one of them;
 \item[(iii)] if $\rho \in (1,+\infty)$, there is a competition of production factors without complete expusion of one of them;
 \item[(iv)]  if $\rho\to+\infty$, we obtain the Leontief fixed proportion function $f(v_1,\dots,v_n) = \min(v_1,\dots,v_n)$, and there are no substitution between production factors.
\end{enumerate}
In the limit case (iv) the generalized resource distribution model, which we describe below, is reduced to the classical Houthakker-Johansen model. 

We consider the following generalization of the resource distribution problem \eqref{hjm.prmax}--\eqref{hjm.pru} for maximization of the total output of the industry:
\begin{gather}
  \int_{\mathbb R^n_+} \min \bigl(1,f(\tfrac{u_1(x)}{x_1},\dots,\tfrac{u_n(x)}{x_n}) \bigr) \mu(dx) \to \max_{u(x)}, \label{ghj.prmax} \\
  \int_{\mathbb R^n_+} u_j(x) \mu(dx) \leq l_j, \quad j = 1,\dots, n, \label{ghj.prres}\\
  u(x) = \bigl( u_1(x), \dots, u_n(x) \bigr) \geq 0. \label{ghj.pru}
\end{gather}
Here $u(x) = \bigl( u_1(x),\dots, u_n(x) \bigr)$ is the vector of inputs for a unit capacity using technology $x$.

Consider the Young transform of function $f(v)$:
\begin{equation}
  h(p) = \inf \{ \tfrac{p\cdot v}{f(v)} \mid v \geq 0,\; f(v) > 0 \}.
\end{equation}
Denote 
\begin{equation}
 p \circ x = (p_1 x_1, \dots, p_n x_n).
\end{equation}
The unit cost of production using technology $x$ for the unit prices $p$ of inputs is equal to $h( p \circ x )$. The unit capacity profit function for technology $x$ is $\pi(x,p,p_0) = (p_0 - h(p \circ x) )_+$, where $(\cdot)_+$ is defined in \eqref{hjm.a+}.

\begin{theorem}[see \cite{Shananin1997ena}]\label{ghj.thmGNP} Let $\mu$ be a finite compactly supported non-negative Borel measure in $\mathbb R^n_+$. The following statements are valid:
\begin{itemize}
 \item[(i)] If $l \geq 0$, then problem \eqref{ghj.prmax}--\eqref{ghj.pru} has a solution in the class of vector-functions $u(x)$ with $\mu$-integrable components.
 \item[(ii)] Let vector-function $u(x)$ satisfying \eqref{ghj.prres}, \eqref{ghj.pru} be a solution of problem \eqref{ghj.prmax}--\eqref{ghj.pru}. Then there exist $p_0 \geq 0$, $p = (p_1,\dots,p_n)\geq 0$, $p_0+|p|>0$, such that
 \begin{gather}
	p_j \biggl( l_j - \int_{\mathbb R^n_+} u_j(x) \mu(dx) \biggr) = 0 \quad (j = 1,\dots, n),\label{ghj.gNP1} \\
	\text{$u(x) = 0$ for $\mu$-almost all $x \in \mathbb R^n_+$ such that $p_0 < h(p \circ x)$},\label{ghj.gNP2}\\
	\begin{gathered}\label{ghj.gNP3}
	\text{$f(u(x)) = 1$, $p_0 - pu(x) = \pi(x,p,p_0)$ for}\\
	\text{$\mu$-almost all $x \in \mathbb R^n_+$ such that $p_0 > h(p \circ x)$}.
	\end{gathered}
 \end{gather}
  \item[(iii)] Let $u(x)$ satisfy \eqref{ghj.prres}, \eqref{ghj.pru} and let $l = (l_1,\dots,l_n) > 0$. If there exist $p_0 \geq 0$, $p = (p_1,\dots,p_n)\geq 0$, $p_0+|p|>0$, satisfying \eqref{ghj.gNP1}--\eqref{ghj.gNP3}, then $u(x)$ is a solution of problem \eqref{ghj.prmax}--\eqref{ghj.pru}.
\end{itemize}
\end{theorem}

In a similar way with Lemma \ref{hjm.lemmaGNP}, Lagrange multipliers $p_0$, $p = (p_1,\dots,p_n)$ are interpreted as unit prices of output and inputs, respectively. Then Theorem \ref{ghj.thmGNP} states that the optimal resource distribution mechanism is equivalent to a market-type mechanism: profitable technologies for prices $p_0 \geq 0$, $p = (p_1,\dots,p_n) \geq 0$, are used completely, unprofitable technologies are not used. Besides, $p_0 \geq 0$, $p = (p_1,\dots,p_n) \geq 0$ are the equilibrium prices: they are determined by the condition of equilibrium of demand and supply on the market of resources.

We define the production function $F(l)$ and the profit function $\Pi(p,p_0)$ for the resource distribution problem \eqref{ghj.prmax}--\eqref{ghj.pru} as follows.

\begin{definition}\label{ghj.Fdef} Function $F(l)$ which assigns to each $l = (l_1,\dots,l_n) \geq 0$ the optimal value of functional in the resource distribution problem \eqref{ghj.prmax}--\eqref{ghj.pru} is called \textit{production function} for this problem.
\end{definition}

\begin{definition}
Function $\Pi(p,p_0)$ given by
\begin{equation}\label{ghj.profdef}
  \Pi(p,p_0) = \int_{\mathbb R^n_+}  ( p_0 - h(p \circ x) )_+ \mu(dx).
\end{equation}
is called \textit{profit function} for the resource distribution problem \eqref{ghj.prmax}--\eqref{ghj.pru}.
\end{definition}

Is it shown in \cite{Shananin1997enb} that the production function for the resource distribution problem \eqref{ghj.prmax}--\eqref{ghj.pru} is neoclassical: it is concave, non-decreasing and continuous in $\mathbb R^n_+$. It is also shown in \cite{Shananin1997enb} that the production and profit functions for problem \eqref{ghj.prmax}--\eqref{ghj.pru} are related by formulas \eqref{hjm.FtoPi}, \eqref{hjm.PitoF}.

\subsection{Connection between micro- and macro-level descriptions and new problems of integral geometry}\label{ghj.inv}

In a similar way with Subsection \ref{hjm.mic-mac}, we are interested in study of connection between distributions $\mu(dx)$ of capacities over technologies and corresponding production functions $F(l)$. One of motivations is the problem of aggregation described in Subection \ref{agr.pro} and the problem of universality of a resource distribution model stated in Subsection \ref{agr.uni} in the case of the Houthakker-Johansen model.

In view of \eqref{hjm.FtoPi}, \eqref{hjm.PitoF}, study of this relation is equivalent to investigation of operator \eqref{ghj.profdef} and transforms \eqref{hjm.FtoPi}, \eqref{hjm.PitoF}.

In a similar way with \eqref{hjm.PitoRad}, integral operator \eqref{ghj.profdef} is related to generalized Radon transform. In the case of absolutely continuous measures $\mu(dx) = f(x) dx$ this relation is given by
\begin{equation}\label{ghj.PitoRad}
  \begin{gathered}
	\frac{\partial^2 \Pi(p,p_0)}{\partial p_0^2} = \int\limits_{h(p \circ x)=p_0} f(x) \Omega(p,x), \quad \mu(dx) = f(x) dx, \\
	 d_x h(p \circ x) \wedge \Omega(p,x) = dx_1 \wedge \cdots \wedge dx_n.
  \end{gathered}
\end{equation}
The differential $(n-1)$-form $\Omega(p,x)$ is sometimes called the \textit{Gelfand-Leray form}. In view of \eqref{ghj.PitoRad}, study of operator $\Pi(p,p_0)$ is equivalent to study of the generalized Radon transform over the level hypersurfaces
\begin{equation}
  \{ x \in \mathbb R^n_+ \mid h(p \circ x) = p_0 \}, \quad p \in \mathbb R^n_+, \; p_0 > 0.
\end{equation}

Next, we recall a criterion for injectivity of operator \eqref{ghj.profdef}. Put 
\begin{equation}
  \widehat h(z) = \int_{\mathbb R^n_+} x_1^{z_1-1} \dots x_n^{z_n-1} e^{-h(x)} \, dx, \quad \Re z = ( \Re z_1,\dots,\Re z_n) > 0.
\end{equation}

\begin{theorem}[See \cite{Agaltsov2016c}] \label{ghj.thminj} Let $h(x)$ be such that
\begin{equation}
  \begin{gathered}\label{ghj.hprop}
	h(x) \in C^1(\mathop{\mathrm{int}}\mathbb R^n_+), \; \text{$h(x) > 0$  and $h(\lambda x) = \lambda h(x)$ for $\lambda > 0$, $x \in \mathbb R^n_+$}, \\
	\text{the level sets of $h(x)$ are bounded}.
  \end{gathered}
\end{equation}
Consider the following operator:
\begin{equation}\label{ghj.profop}
  f(x) \mapsto \Pi(p,p_0), \;\; \text{where $\Pi(p,p_0)$ is given by \eqref{ghj.profdef} with $\mu(dx)=f(x) dx$}.
\end{equation}
The following statements hold:
\begin{enumerate}
 \item[(i)] Operator \eqref{ghj.profop} is injective in $L^\infty(\mathbb R^n_+)$ if and only if $\widehat h(z) \neq 0$ for any $z \in \mathbb C^n$, $\Re z = (1,\dots,1)$.
 \item[(ii)] Operator \eqref{ghj.profop} is injective in $L^2(\mathbb R^n_+)$ if and only if $\widehat h(z) \neq 0$ for almost all $z \in \mathbb C^n$, $\Re z = (\tfrac 1 2,\dots,\tfrac 1 2)$.
\end{enumerate}
\end{theorem}
It is also shown in \cite{Agaltsov2016c} that for any $c \in \mathbb R^n_+$, non-vanishing of $\widehat h(z)$ for $z \in \mathbb C^n$, $\Re z = c$, characterizes injectivity of \eqref{ghj.hprop} in a certain weighted $L^p$ space.

The proof of Theorem \ref{ghj.thminj} is based on appropriate multidimensional generalizations of Wiener's Tauberian theorems.

\begin{example}[See \cite{Agaltsov2016c} for details]
Theorem \ref{ghj.thminj} establishes uniqueness of the optimal resource distribution for the resource distribution problem \eqref{ghj.prmax}--\eqref{ghj.pru} for \textit{nested bounded level set CES} unit cost functions $h(x)$, which are defined as follows:
\begin{enumerate}
 \item[(i)] Function $h(x) = C ( a_1 x_1^{-r}+\cdots+ a_n x_n^{-r} )^{-\frac 1 r}$, $r \in [-1,0)$, $C$, $a_1$, \dots, $a_n > 0$, is a nested bounded level set CES function.
 \item[(ii)] If $h_1(x_1,\dots,x_n)$, $h_2(y_1,\dots,y_m)$ are nested bounded level set CES functions and $i \in \{1,\dots,n\}$, then function
 \begin{equation}
  \begin{gathered}
	h(x_1,\dots,x_{n-1},y_1,\dots,y_m) \\
	= h_1\bigl(x_1,\dots,x_{i-1},h_2(y_1,\dots,y_m),x_i,\dots,x_{n-1}\bigr)
	\end{gathered}
 \end{equation}
 is a nested bounded level set CES function. 
\end{enumerate}
Nested CES functions were introduced in \cite{Sato1967} as a generalization of CES functions, allowing different elasticities of substitution in different groups of production factors. In particular, nested CES functions allow to take into account separability of groups of production factors and the difference of elasticities of substitution in these groups.
\end{example}

The following theorem of \cite{Agaltsov2015d} establishes a relation between profit function \eqref{ghj.profdef} and the Laplace transform of the distribution of capacities.

Put 
\begin{equation}\label{ghj.Gdef}
  \begin{gathered}
   G(s) = (2\pi i)^{-n} \int_{c+i\mathbb R^n} s_1^{-z_1}\dots s_n^{-z_n} \rho_h(z) \biggl( \int_{\mathbb R^n_+} p^{z-I} \Pi(p,1)  dp \biggr) \, dz_1 \dots dz_n, \\
   \rho_h(z) = \widehat h(z)^{-1}\Gamma(z_1)\cdots\Gamma(z_n)\Gamma(2+z_1+\cdots+z_n) , \quad p^{z-I}=p_1^{z_1-1}\cdots p_n^{z_n-1}.
   \end{gathered}
\end{equation}

\begin{theorem}[See \cite{Agaltsov2015d}]\label{ghj.thmchar} Let $h(x)$ satisfy \eqref{ghj.hprop} and let $\rho_h(z) \in L^2(c+i\mathbb R^n) \cap L^\infty(c+i\mathbb R^n)$ for some $c \in \mathop{\mathrm{int}} \mathbb R^n_+$, where $\rho_h(z)$ is the function of \eqref{ghj.Gdef}. Let $\mu$ be a non-negative Borel measure on $\mathbb R^n_+$ such that 
\begin{equation}
  \int_{\mathbb R^n_+} x_1^{-c_1}\dots x_n^{-c_n} \mu(dx) < \infty, \quad \int_{\mathbb R^n_+} x_1^{-2c_1}\dots x_n^{-2c_n} \mu(dx) < \infty.
\end{equation}
A function $\Pi(p,p_0)$ can be represented in the form \eqref{ghj.profdef} if and only if
\begin{gather}
	\int_{\mathbb R^n_+} p_1^{c_1-1}\dots p_n^{c_n-1} |\Pi(p,1)| \, dp < \infty, \;\; \int_{\mathbb R^n_+} p_1^{2c_1-1}\dots p_n^{2c_n-1} |\Pi(p,1)|^2 \, dp < \infty,\\
	\Pi(p_1,\dots,p_n,p_0) = p_0 \Pi(\tfrac{p_1}{p_0},\dots,\tfrac{p_n}{p_0},1), \quad (p_1,\dots,p_n,p_0)>0,\\
	G(s) = \int_{\mathbb R^n_+} e^{-s\cdot x} \mu(dx), \; s \in \mathbb R^n_+, \;\; \text{where $G(s)$ is the function of \eqref{ghj.Gdef}}
\end{gather}
\end{theorem}

Thus, in view of Theorem \ref{ghj.thmchar}, the question of existence of a distribution of capacities corresponding to given production (or profit) function is reduced to checking whether function $G(s)$ defined by \eqref{ghj.Gdef} is the Laplace transform of some measure. This checking can be done using Bernstein-Bochner theorem on completely monotone functions (see \cite{Bochner1955}), or, what is more convenient, using local characterization conditions \eqref{hjm.Lapchar}, see \cite{Henkin1990a} for details.

On the other hand, in general, it is very difficult to find an analytic expression for function \eqref{ghj.Gdef}. In the following proposition we give a simple criterion for existence of a distribution of capacities for a given profit function in a particular case.

\begin{proposition}\label{ghj.propchar} Let $r \in [-1,0)$ and let $h(p \circ x)$ be defined by
\begin{equation}\label{ghj.hrdef}
  h(p \circ x) = \bigl( p_1^{-r} x_2^{-r} + \cdots + p_n^{-r} x_n^{-r} \bigr)^{-\frac 1 r}.
\end{equation}
Let $f \in C(\mathbb R^n_+)$ and suppose that
\begin{equation}
	\text{for any $A>0$ there exists $C_A>0$ such that $|f(x)| \leq C_A \cdot e^{A|x|}$, $x \in \mathbb R^n_+$}.\label{ghj.fdef}
\end{equation}
A function $\Pi(p,p_0)$ can be represented in the form 
\begin{equation}
  \begin{gathered}\label{ghj.phidef}
	\Pi(p,p_0) = \int_{\mathbb R^n_+} \bigl( p_0 - h(p \circ x) )_+ \varphi(x) \, dx, \\
	\varphi(x) = (-r)^{n-1} (x_1 \dots x_n)^{-r - 1} \int_0^{+\infty} t^{n-1} e^{-t} f( t x_1^{-r}, \dots, t x_n^{-r}) dt,
  \end{gathered}
\end{equation}
where $(\cdot)_+$ is defined in \eqref{hjm.a+}, if and only if
\begin{gather}
  \Pi(p,+0) = \tfrac{\partial \Pi}{\partial p_0}(p,+0) = 0, \quad p = (p_1,\dots,p_n) > 0, \label{ghj.Pilim}\\
  \Pi(p_1,\dots,p_n,p_0) = p_0 \Pi( \tfrac{p_1}{p_0},\dots,\tfrac{p_n}{p_0}, 1), \quad (p_1,\dots,p_n,p_0) > 0, \label{ghj.Pihomo}\\
  \begin{gathered}\label{ghj.d2Pi}
  \tfrac{\partial^2 \Pi}{\partial p_0^2} (p_1^{-\frac 1 r},\dots,p_n^{-\frac 1 r},1) = \int_{\mathbb R^2_+} e^{-p \cdot x} f(x) \, dx, \quad p = (p_1,\dots,p_n) > 0.
	\end{gathered}
\end{gather}
\end{proposition}
The proof of Proposition \ref{ghj.propchar} is given in Appendix \ref{ape.ghj}.

In the following examples we show that the Cobb-Douglas and CES production functions admit different distributions of capacities corresponding to different elasticities of substitution at the micro-level, and we compute explicitly these distributions using Proposition \ref{ghj.propchar}.

\begin{example}\label{ghj.exCD} Consider the Cobb-Douglas production function $F_\text{CD}(l_1,l_2)$ defined in \eqref{hjm.FCDdef}. Let $r \in [-1,0)$. Then $F_\text{CD}(l_1,l_2)$ is the production function for the resource distribution problem \eqref{ghj.prmax}--\eqref{ghj.pru} with 
\begin{gather}
  h( p \circ x ) = ( p_1^{-r} x_1^{-r} + p_2^{-r} x_2^{-r} )^{-\frac 1 r}, \label{ghj.PiCDh}\\
  \begin{gathered}\label{ghj.PiCDmu}
  \mu(dx) = A_r \cdot x_1^{\alpha_1-1} x_2^{\alpha_2-1} dx, \\ A_r = (-r) A \tfrac{B(\alpha_1,\alpha_2) }{B(-\frac{\alpha_1}{r},-\frac{\alpha_2}{r})},
  \end{gathered}
\end{gather}
where $B(\cdot,\cdot)$ is the beta function. Demonstration is given in Appendix \ref{ape.ghj}.
\end{example}

\begin{example}\label{ghj.exCES} Consider the CES production function $F_\text{CES}(l_1,l_2)$ defined in \eqref{hjm.FCES} with $\rho \in (0,1)$. It is shown in \cite{Henkin1990a} that $F_\text{CES}(l_1,l_2)$ is the production function for the resource distribution problem \eqref{hjm.prmax}--\eqref{hjm.pru} without substitution at the micro-level. Besides, $F_\text{CES}(l_1,l_2)$ is the production function for the resource distribution problem \eqref{ghj.prmax}--\eqref{ghj.pru} with 
\begin{gather}
	h(p \circ x) = \bigl( (p_1 x_1)^{\frac{2\rho}{1+\rho}}+(p_2 x_2)^{\frac{2\rho}{1+\rho}}\bigr)^{\frac{1+\rho}{2\rho}}, \label{ghj.PiCESh} \\
	\begin{gathered}\label{ghj.PiCESmu}
    \mu(dx) = B_\rho \cdot (x_1 x_2)^{-\frac{\rho}{1+\rho} -1} \bigl( \alpha_1^{\frac{2}{1+\rho}}x_1^{-\frac{2\rho}{1+\rho}} + \alpha_2^{\frac{2}{1+\rho}} x_2^{-\frac{2\rho}{1+\rho}} \bigr)^{-\frac{\gamma}{1-\gamma}\frac{1+\rho}{2\rho}-1}, \\
    B_\rho = \tfrac{\gamma^{\frac{1}{1-\gamma}}}{1-\gamma}  \tfrac{2^{b-2}}{\pi} (\alpha_1 \alpha_2)^{\frac{1}{1+\rho}} bB(\tfrac b 2, \tfrac b 2), \quad b = \tfrac{\gamma}{1-\gamma} \tfrac{1+\rho}{\rho}.
    \end{gathered}
\end{gather}
where $B(\cdot,\cdot)$ is the beta function. Demonstration is given in Appendix \ref{ape.ghj}. 
\end{example}

\section{Identification of production model}\label{sec.ide} We consider the problem of identification for production models in contemporary national economies. As it was noted in Section \ref{sec.int}, globalization of the world economy leads to competition of home-made goods with their imported analogs on national domestic markets.

For example, consider the case of contemporary Russian economy. During the period when the Central Bank of Russia maintains the stable ruble exchange rate, imported goods dominate and substitute home-made products. The reason is that the inflation rate in Russia is higher than in developed capitalist countries. The opposite situation, when domestic goods dominate and substitute the imported products occurs after economic crises leading to ruble weakening. These periods have been interwining in the last decades, facilitating standartization. As a result, production have become adapted to market conditions due to substitutability of production factors at the micro-level. 

In a similar way with Section \ref{sec.ghj}, we describe substitutability of production factors at the micro-level using the unit cost function
\begin{equation}\label{ide.hdef}
  \begin{gathered}
  h(p \circ x) = h( p_1 x_1, \dots, p_n x_n),\\
  \text{$h(x) > 0$ and $h(\lambda x) = \lambda h(x)$ for $\lambda>0$, $x \in \mathbb R^n_+$}.
  \end{gathered}
\end{equation}
In general, in economic statistics only the data of production volumes and prices is available. We use the following notations:
\begin{equation}\label{ide.tsas}
 \begin{aligned}
	& \text{$y(t) \geq 0$: output of the industry at time $t$}, \\
	& \text{$p_0(t) > 0$: unit price of the output at time $t$},\\
	& \text{$p(t) = (p_1(t),\dots,p_n(t)) > 0$: unit prices of CUPF at time $t$}.\\
 \end{aligned}
\end{equation}
Observable time series $\{ y(t), p_0(t), p(t) \mid t = 1,\dots,T \}$ are \textit{compatible with the unit cost function} $h(p \circ x)$ if the following \textit{moment problem} is solvable:
\begin{problem}[Moment problem]\label{ide.mompr} Given time series $\{ y(t), p_0(t), p(t) \mid t = 1,\dots,T \}$ satisfying \eqref{ide.tsas}, find a non-negative absolutely continuous measure $\mu(dx)$ such that
\begin{equation}\label{ide.mompreq}
  \int_{\mathbb R^n_+} \theta\bigl( p_0(t) - h( p(t) \circ x ) \bigr) \mu(dx) = y(t) \quad (t = 1,\dots,T),
\end{equation}
where $\theta$ is defined in \eqref{hjm.thetadef}, and $h(x)$ satisfies \eqref{ide.hdef}.
\end{problem}

A criterion for solvability of Problem \ref{ide.mompr} was proposed in \cite{Shananin1999}. In order to formulate this criterion, we need to introduce some notations. 

Hypersurfaces
\begin{equation}\label{ide.hypers}
  \bigl\{ x \mid h (p(t) \circ x ) = p_0(t) \bigr\} \quad (t=1,\dots,T),
\end{equation}
determine a partition of $\mathbb R^n_+$. This partition depends only on vectors $\widehat p(t) = \tfrac{1}{p_0(t)} p(t)$, $t = 1$, \dots, $T$, since function $h(p \circ x)$ is positive homogeneous according to \eqref{ide.hdef}. Denote by $\Lambda_h(\widehat p) = \Lambda_h\{ \widehat p(t) \mid t = 1,\dots,T \}$ the set of domains of this partition:
\begin{equation}\label{ide.Lamdef}
  \begin{gathered}
  \Lambda_h(\widehat p) = \Lambda_h\{ \widehat p(t) \mid t = 1,\dots,T \} = \bigl\{ G_S \mid S \subset \{1,\dots,T\} \bigr\}, \\
  G_S = \bigl\{ x \in \mathop{\mathrm{int}} \mathbb R^n_+ \mid  \text{$h(p(t) \circ x) > p_0(t)$ for any $t \in S$}  \\
  \text{and $h(p(t) \circ x) < p_0(t)$ for any $t \in \{1,\dots,T\} \setminus S$}  \bigr\}.
  \end{gathered}
\end{equation}

For each domain $G \in \Lambda_h(\widehat p)$ we define its \textit{spectrum} $Z(G)$ as follows:
\begin{equation}\label{ide.specdef}
  \begin{gathered}
	Z(G) = \bigl( Z_1(G), \dots, Z_T(G) \bigr),\quad G \in \Lambda_h( \widehat p ),\\
	Z_t(G) = \begin{cases}
				\text{$1$, if $p_0(t) > h (p(t) \circ x)$ for any $x \in G$},\\
				\text{$0$, if $p_0(t) < h(p(t) \circ x)$ for any $x \in G$},
	         \end{cases}\quad (t = 1,\dots,T).
  \end{gathered}
\end{equation}
Next,
\begin{equation}\label{ide.Gdef}
  \begin{gathered}
  \text{let $\Gamma_h(\widehat p) = \Gamma_h\{ \widehat p(t) \mid t =1,\dots,T \}$ be the convex cone in $\mathbb R^T$}\\
  \text{spanned by} \; \{ Z(G) \mid G \in \Lambda_h( \widehat p ) \}.
  \end{gathered}
\end{equation}

\begin{proposition}[see \cite{Shananin1999}]\label{ide.propsol} Problem \ref{ide.mompr} is solvable if and only if
\begin{equation}\label{ide.yhinG}
  y = ( y(1),\dots,y(T) ) \in \Gamma_h(\widehat p).
\end{equation}
\end{proposition}
The following proposition establishes an easily verifiable necessary condition for solvability of Problem \ref{ide.mompr}. Put
\begin{equation}\label{ide.wdef}
 \begin{gathered}
	w(t) = \bigl( w_G(t) \mid G \in \Lambda_h( \widehat p ) \bigr) \quad (t = 1,\dots,T),\\
	w_G(t) = \begin{cases}
				1, & \text{if $Z_t(G) = 1$},\\
				0, & \text{if $Z_t(G) = 0$}.
	         \end{cases}
 \end{gathered}
\end{equation}

\begin{proposition}[Necessary condition for solvability of Problem \ref{ide.mompr}] \label{ide.propnes} Let $y(t)$, $w(t)$ be the same as in \eqref{ide.mompreq}, \eqref{ide.wdef}. A necessary condition for solvability of  Problem \ref{ide.mompr} is that for any $\Omega_1$, $\Omega_2 \subseteq \{1,\dots,T\}$ the inequality
\begin{equation}\label{ide.w1>w2}
  \sum_{t \in \Omega_1} w(t) \geq \sum_{t \in \Omega_2} w(t),
\end{equation}
implies the inequality
\begin{equation}\label{ide.y1>y2}
  \sum_{t \in \Omega_1} y(t) \geq \sum_{t \in \Omega_2} y(t).
\end{equation}
\end{proposition}
Proposition \ref{ide.propnes} is proved in Appendix \ref{ape.ide}.

Efficient application of Proposition \ref{ide.propsol} to identification of the generalized Houthakker-Johansen requires investigation of properties of polyhedral cone $\Gamma_h (\widehat p)$.

\begin{proposition}\label{ide.propintnorm} Each face of cone $\Gamma_h (\widehat p )$ admits a non-zero normal vector with integer coordinates.
\end{proposition}
Proposition \ref{ide.propintnorm} is proved in Appendix \ref{ape.ide}.

The following proposition shows that the necessary condition of Proposition \ref{ide.propnes} is also sufficient for solvability of Problem \ref{ide.mompr} if and only if each face of cone $\Gamma_h ( \widehat p )$ admits a normal vector with coordinates in $\{-1,0,1\}$. 

Let $\Gamma$ be a polyhedral cone, i.e. a convex cone spanned by a finite set of points. The dual cone $\Gamma^*$ is defined according to \eqref{agr.K*def}. The edges of the dual cone $\Gamma^*$ are normals to faces of $\Gamma$.

\begin{definition}\label{ide.defdisconv} Polyhedral cone $\Gamma$ is called \textit{discretely convex} if polyhedrons
\begin{align*}
  \Gamma  & \cap \{ y = (y_1,\dots,y_T) \mid -1 \leq y_j \leq 1, \; j = 1,\dots, T\},\\
  \Gamma^* & \cap \{ y = (y_1,\dots,y_T) \mid -1 \leq y_j \leq 1, \; j = 1,\dots,T\}
\end{align*}
are discretely convex, i.e. all their vertices have coordinates in $\{-1,0,1\}$. Here $\Gamma^*$ denotes the dual cone to $\Gamma$, defined according to \eqref{agr.K*def}.
\end{definition}
One can see that polyhedral cone $\Gamma$ is discretely convex if and only if:
  \begin{enumerate}
   \item[(i)] $\Gamma$ is spanned by a finite set of vectors with coordinates in $\{-1,0,1\}$;
   \item[(ii)] each face of $\Gamma$ admits a non-zero normal vector with coordinates in $\{-1,0,1\}$.
  \end{enumerate}

\begin{proposition}\label{ide.propsuf} The necessary condition of Proposition \ref{ide.propnes} is also  sufficient for solvability of Problem \ref{ide.mompr} if and only if the polyhedral cone $\Gamma_h (\widehat p)$ is discretely convex.
\end{proposition}
Proposition \ref{ide.propsuf} is proved in Appendix \ref{ape.ide}.

\begin{example} Consider the CES unit cost function $h(p \circ x)$ given by
\begin{equation}\label{ide.CESdef}
  h(p \circ x) = \bigl(p_1^{-\rho} x_1^{-\rho} + p_2^{-\rho} x_2^{-\rho}\bigr)^{-\frac 1 \rho}, \quad \rho \in [-1,0) \cup (0,+\infty).
\end{equation}
It is shown in \cite{Molchanov2013a} that in this case cone $\Gamma_h(\widehat p) = \Gamma_h\{ \widehat p(t) \mid t = 1,\dots, T\}$ is always discretely convex for $T \leq 5$. However, for any $T \geq 6$ one can find a cone $\Gamma_h(\widehat p) = \Gamma_h\{ \widehat p(t) \mid t=1,\dots,T\}$ which is not discretely convex. See \cite{Molchanov2013a} for more details.
\end{example}

\section{Estimation of micro-level elasticity of substitution and related combinatorial problems}\label{sec.ele}

\subsection{The problem of estimation of micro-level elasticity of subsitution}\label{ele.est}

We consider the problem of estimation of elasticity of substitution for the CES unit cost function \eqref{ide.CESdef} at the micro-level. Recall that $\rho$ is related to elasticity of substitution $\sigma$ by $\sigma = \tfrac{1}{1+\rho}$.

\begin{problem}\label{ele.mompr} Given the time series $\{ y(t), p_0(t), p(t) \mid t = 1,\dots,T\}$ of outputs and prices satisfying \eqref{ide.tsas}, find the values of parameter $\rho \in [-1,0) \cup (0,+\infty)$ for which the moment problem
\begin{equation}\label{ele.mompr1}
  \int_{\mathbb R^2_+} \theta \bigl( p_0(t) -  \bigl( p_1(t)^{-\rho} x_1^{-\rho} + p_2(t)^{-\rho} x_2^{-\rho}  \bigr)^{-\frac 1 \rho} \bigr) \mu(dx) = y(t), \quad (t = 1,\dots,T)
\end{equation}
is solvable in the class of non-negative absolutely continuous measures $\mu(dx)$.
\end{problem}

According to Proposition \ref{ide.propsol}, $\rho$ is a solution of Problem \ref{ele.mompr} if and only if 
\begin{gather}
  y = \bigl( y(1), \dots, y(T) \bigr) \in \Gamma_\rho(\widehat p), \notag \\
  \Gamma_\rho(\widehat p) \overset{def}{=\joinrel=} \Gamma_h(\widehat p) = \Gamma_h\{\widehat p(t) \mid t = 1,\dots,T\}, \quad \text{$h$ is defined in \eqref{ide.CESdef}}, \label{ele.Grhodef}\\
  \widehat p(t) = ( \widehat p_1(t), \widehat p_2(t) ), \quad \widehat p_1(t) = \tfrac{p_1(t)}{p_0(t)}, \quad \widehat p_2(t) = \tfrac{p_2(t)}{p_0(t)}. \label{ele.phatdef}
\end{gather}
where $\Gamma_h(\widehat p) = \Gamma_h\{\widehat p(t) \mid t = 1,\dots,T\}$ is defined in \eqref{ide.Gdef}.

The lines 
\begin{equation}\label{ele.part}
  \bigl( \widehat p_1(t)^{-\rho} x_1^{-\rho} + \widehat p_2(t)^{-\rho} x_2^{-\rho} \bigr)^{-\frac 1 \rho} = 1 \quad (t = 1,\dots,T),
\end{equation}
determine a partition of $\mathbb R^2_+$. The cone $\Gamma_\rho (\widehat p)$ is invariant with respect to variations of parameter $\rho$ which do not change the set of spectral vectors \eqref{ide.specdef} of this partition. Variation of parameter $\rho$ can affect solvability of the moment problem \eqref{ele.mompr1}  if one of the lines of \eqref{ele.part} passes through the intersection point of two other lines. Besides, the three lines
\begin{equation}
 \begin{gathered}
	\bigl( \widehat p_1(t_1)^{-\rho} x_1^{-\rho} + \widehat p_2(t_1)^{-\rho} x_2^{-\rho} \bigr)^{-\frac 1 \rho} = 1,\\
	\bigl( \widehat p_1(t_2)^{-\rho} x_1^{-\rho} + \widehat p_2(t_2)^{-\rho} x_2^{-\rho} \bigr)^{-\frac 1 \rho} = 1,\\
	\bigl( \widehat p_1(t_3)^{-\rho} x_1^{-\rho} + \widehat p_2(t_3)^{-\rho} x_2^{-\rho} \bigr)^{-\frac 1 \rho} = 1,\\
 \end{gathered}
\end{equation}
have a common intersection point if and only if $\rho$ is a solution of equation
\begin{equation}\label{ele.detsys}
 \begin{vmatrix}
	1 & 1 & 1 \\
	\widehat p_1(t_1)^{-\rho} & \widehat p_1(t_2)^{-\rho} & \widehat p_1(t_3)^{-\rho} \\
	\widehat p_2(t_1)^{-\rho} & \widehat p_2(t_2)^{-\rho} & \widehat p_2(t_3)^{-\rho}
 \end{vmatrix} = 0.
\end{equation}

\begin{proposition}[see \cite{Molchanov2013b}]\label{ele.3lines} For any $1 \leq t_1 < t_2 < t_3 \leq T$ there exists at most one $\rho \in [-1,0)\cup(0,+\infty)$ satifsying \eqref{ele.detsys}. 
\end{proposition}

According to Proposition \ref{ele.3lines}, the values of $\rho \in [-1,0) \cup (0,+\infty)$ for which there is a triple intersection point for the lines of family \eqref{ele.lines}, divide the set $\rho \in [-1,0) \cup (0,+\infty)$ into at most ${T \choose 3} + 2$ subintervals, where ${ T \choose 3} = \tfrac{T!}{3!(T-3)!}$ is the binomial coefficient. Variations of $\rho$ within each of these subintervals do not affect solvability of the moment problem \eqref{ele.mompr}. Thus, combining Propositions \ref{ide.propsol} and \ref{ele.3lines}, we have the following corollary. 

\begin{corollary} Finding of $\rho \in [-1,0) \cup (0,+\infty)$ for which the moment problem \eqref{ele.mompr} is solvable requires at most ${T \choose 3}$ operations to solve equations \eqref{ele.detsys} and ${T \choose 3} + 2$ operations to check whether vector $y = ( y(1),\dots, y(T))$ belongs to polyhedral cone $\Gamma_\rho(\widehat p) = \Gamma_\rho \{ \widehat p(t) \mid t = 1,\dots, T \}$. 
\end{corollary}

Thus, the problem of determination of elasticity of substitution of production factors is solvable in polynomial time with respect to the length $T$ of the time series.

\subsection{Simplification of Problem \ref{ele.mompr}}\label{ele.sim}
In order to study Problem \ref{ele.mompr}, it is convenient to make an appropriate change of variables in \eqref{ele.mompr1}.

If $\rho \in [-1,0)$ we make the following change of variables in formula \eqref{ele.mompr1}:
\begin{equation}\label{ele.change<0}
\begin{gathered}
  z_1 = x_1^{-\rho}, \quad z_2 = x_2^{-\rho},\\
  \widetilde p_1(t,\rho)= \widehat p_1(t)^{-\rho}, \quad \widetilde p_2(t,\rho)= \widehat p_2(t)^{-\rho},
\end{gathered}
\end{equation}
where $\widehat p(t)$ is defined in \eqref{ele.phatdef}.

Now suppose that $\rho \in (0,+\infty)$. Consider the following change of variables:
\begin{equation}\label{ele.change>0}
\begin{gathered}
z_1 = \frac{\varepsilon x_1^{-\rho}}{x_1^{-\rho}+x_2^{-\rho}-\varepsilon}, \quad 
z_2 = \frac{\varepsilon x_2^{-\rho}}{x_1^{-\rho}+x_2^{-\rho}-\varepsilon}, \\
\widetilde p_1(t,\rho) = \tfrac 1 \varepsilon - \widehat p_1(t)^{-\rho}, \quad
\widetilde p_2(t,\rho) = \tfrac 1 \varepsilon - \widehat p_2(t)^{-\rho},
\end{gathered}
\end{equation}
for some fixed $\varepsilon > 0$. Note that the map $(x_1,x_2) \to (z_1,z_2)$ maps the domain $\{(x_1,x_2) \in \mathbb R^2_+ \mid x_1^{-\rho}+x_2^{-\rho} > \varepsilon\}$ diffeomorphically to $\mathop{\mathrm{int}}\mathbb R^2_+$. Also note that
\begin{equation}
  1 - \widetilde p_1(t,\rho) z_1 - \widetilde p_2(t,\rho) z_2 = \varepsilon \frac{\widehat p_1(t)^{-\rho} x_1^{-\rho} + \widehat p_2(t)^{-\rho} x_2^{-\rho} - 1}{x_1^{-\rho}+x_2^{-\rho} - \varepsilon}.
\end{equation}
In particular,
\begin{equation}\label{ele.thetaseq}
\begin{gathered}
  \theta\bigl( 1 - ( \widehat p_1(t)^{-\rho} x_1^{-\rho} + \widehat p_2(t)^{-\rho} x_2^{-\rho} \bigr)^{-\frac 1 \rho} \bigr) = \theta\bigl(1 - \widehat p_1(t,p)z_1 - \widehat p_2(t,p) z_2 \bigr), \\
  \text{if $x_1^{-\rho} + x_2^{-\rho} > \varepsilon$}.
\end{gathered}
\end{equation}
Besides, there exists $\varepsilon > 0$ such that
\begin{equation}\label{ele.epsineq}
  \begin{gathered}
  \bigl( \widehat p_1(t)^{-\rho} x_1^{-\rho} + \widehat p_2(t)^{-\rho} x_2^{-\rho}  \bigr)^{-\frac 1 \rho} > 1 \quad \text{if $x_1^{-\rho}+x_2^{-\rho} \leq \varepsilon$},\\
  \widehat p_1(t)^{-\rho} < \tfrac 1 \varepsilon, \quad \widehat p_2(t) < \tfrac 1 \varepsilon \quad (t=1,\dots,T).
  \end{gathered}
\end{equation}
Thus, in view of \eqref{ele.mompr1}, \eqref{ele.epsineq}, without loss of generality, one can suppose that the measure $\mu$ of Problem \ref{ele.mompr} is supported in $\{ (x_1,x_2) \in \mathbb R^2_+ \mid x_1^{-\rho}+x_2^{-\rho} > \varepsilon\}$.

Making the change of variable \eqref{ele.change<0} for $\rho\in[-1,0)$, or the change of variable \eqref{ele.change>0} for $\rho \in (0,+\infty)$, one can reformulate Problem \ref{ele.mompr} as follows.

\begin{problem}\label{ele.momprlines} Given the time series $\{ y(t), p_0(t), p(t) \mid t = 1,\dots,T\}$ of outputs and prices satifsying \eqref{ide.tsas}, find the values of parameter $\rho \in [-1,0) \cup (0,+\infty)$ for which the moment problem
\begin{equation}\label{ele.momprlines1}
  \int_{\mathbb R^2_+} \theta \bigl(1 - \widetilde p_1(t,\rho)z_1 - \widetilde p_2(t,\rho)z_2\bigr) \widetilde\mu(dz) = y(t) \quad (t = 1,\dots,T),
\end{equation}
is solvable in the class of non-negative absolutely continuous measures $\widetilde\mu(dz)$, where $\widetilde p(t,\rho)$ is given by \eqref{ele.phatdef}, \eqref{ele.change<0} for $\rho \in [-1,0)$ and by \eqref{ele.phatdef}, \eqref{ele.change>0} for $\rho \in (0,\infty)$, where $\varepsilon$ is choosen in such a way that \eqref{ele.epsineq} holds.
\end{problem}

Consider the partition of $\mathbb R^2_+$ by the straight lines
\begin{equation}\label{ele.lines}
 \widetilde p_1(t,\rho)z_1 + \widetilde p_2(t,\rho)z_2 = 1 \quad (t = 1,\dots,T).
\end{equation}

Denote by $\Lambda\{ \widetilde p(t,\rho) \mid t = 1,\dots,T \}$ the set of domains of this partition, and by $Z(G)$ the spectral vector for a domain $G$ of this partition:
\begin{equation}\label{ele.spectrum}
   \begin{gathered}
  Z(G) = \bigl(Z_1(G),\dots,Z_T(G)\bigr), \quad G \in \Lambda\{ \widetilde p(t,\rho) \mid t = 1,\dots,T \},\\
  Z_t(G) = \begin{cases}
			  1, & \text{if $\widetilde p_1(t,\rho)z_1 + \widetilde p_2(t,\rho)z_2 < 1$ for any $(z_1,z_2) \in G$}, \\
			  0, & \text{otherwise}.
		  \end{cases}
 \end{gathered}
\end{equation}
Now
\begin{equation}\label{ele.Gdef}
\begin{gathered}
  \text{let $\Gamma\{\widetilde p(t,\rho) \mid t = 1,\dots,T \}$ be the convex cone in $\mathbb R^T$}\\
  \text{spanned by $\{ Z(G) \mid G \in \Lambda\{ \widetilde p(t,\rho) \mid t = 1,\dots,T \} \}$},
  \end{gathered}
 \end{equation}
and let $\Gamma_\rho(\widehat p) = \Gamma_\rho\{ \widehat p(t) \mid t = 1,\dots,T \}$ be the cone defined according to \eqref{ide.Gdef}, \eqref{ele.Grhodef}. It follows from \eqref{ide.specdef}, \eqref{ele.thetaseq}, \eqref{ele.epsineq}, \eqref{ele.spectrum} that the sets of spectral vectors (spectra for short) for partition given by \eqref{ele.lines} and for partition given by \eqref{ide.hypers} coincide. As a corollary, their convex conical hulls also coincide:
\begin{equation}\label{ele.GeqG0}
  \Gamma\{\widetilde p(t,\rho) \mid t = 1,\dots,T \} = \Gamma_\rho( \widehat p ).
\end{equation}
Note that the right-hand side of \eqref{ele.GeqG0} does not depend on $\varepsilon$ of \eqref{ele.change>0} if $\rho>0$. It follows that the cone $\Gamma\{\widetilde p(t,\rho) \mid t = 1,\dots,T \}$ does not depend on $\varepsilon$ as well. 

We recall that a partition of $\mathbb R^2_+$ by a family of curves is called \textit{stretchable} if its spectra coincides with the spectra of some partition of $\mathbb R^2_+$ by straight lines.  Thus partitions of $\mathbb R^2_+$ by the level lines of CES functions are stretchable. However, there are partitions which are not stretchable. We recall that a strict wiring diagram is a set of continuous curves, where each two curves intersect at most once and all intersections are transversal. It is known that all partitions of $\mathbb R^2_+$ by strict wiring diagrams are stretchable for $T \leq 6$, but there exist examples of such partitions which are not stretchable for $T\geq9$. The problem of determining whether a given partition by strict wiring diagram is stretchable is NP-hard. See \cite{Goodman1980,Shor1991} for more details.

\subsection{Formal word associated to partition and its transformations}\label{ele.fwo}
Consider the rays
\begin{equation}\label{ele.ray}
  R_\alpha = \bigl\{ (z_1,z_2) \in \mathbb R^2_+ \mid z_2 = z_1 \tan \alpha \bigr\}, \quad \alpha \in (0,\tfrac \pi 2).
\end{equation}
In addition to \eqref{ide.tsas}, we make the following assumptions:
\begin{gather}
  \text{any three lines of \eqref{ele.lines} do not have a common point}; \label{ele.3pleint}\\
   \begin{gathered}
  \text{each $R_\alpha$ meets at most one intersection of lines of \eqref{ele.lines}},\\
    \begin{gathered}\label{ele.alphas}
	\{\alpha_1,\dots,\alpha_N\} = \{\alpha \in (0,\tfrac \pi 2) \mid \text{$R_\alpha$ meets an intersection of lines of \eqref{ele.lines}}\},\\
	\alpha_0 = 0 < \alpha_1 < \cdots < \alpha_N < \alpha_{N+1} = \tfrac \pi 2.
  \end{gathered}
  \end{gathered}
\end{gather}
Fix $\alpha \in (0,\tfrac \pi 2) \setminus \{\alpha_1,\dots,\alpha_N\}$. As we change $z_1$ from $+\infty$ to $0$, the point $(z_1,z_2(z_1)) \in R_\alpha$ meets the lines of \eqref{ele.lines} in a certain order. We denote this order by $\pi(\alpha) = (\pi_1(\alpha),\dots,\pi_T(\alpha)) \in S_T$:
\begin{equation}\label{ele.pipermut}
  \begin{gathered}
    \text{$R_\alpha$ intersects the $t$-th line of \eqref{ele.lines} at $\bigl(z_1(t),z_2(t)\bigr)$},\\
    z_1( \pi_T(\alpha) ) < \dots < z_1 ( \pi_1(\alpha) ).
  \end{gathered}
\end{equation}
This order does not change if $\alpha$ varies in some subinterval $(\alpha_i,\alpha_{i+1})$, but it changes if $\alpha$ crosses some $\alpha_i$. More precisely, 
\begin{equation}\label{ele.sigma}
  \begin{array}{rcll}
  \pi(\beta) & = & \pi(\alpha), & \alpha, \beta \in (\alpha_i,\alpha_{i+1}), \; 0 \leq i \leq N,\\
  \pi(\beta) & = & \sigma_{t_i} \cdot \pi(\alpha), & \alpha \in (\alpha_{i-1},\alpha_i), \; \beta \in (\alpha_i,\alpha_{i+1}), \; 1\leq i\leq N, \\
  & & & \text{for some transposition $\sigma_{t_i} = (t_i,t_i+1) \in S_T$}.
 \end{array}
\end{equation}
We consider the following formal word:
\begin{equation}\label{ele.word}
  \begin{gathered}
	\omega = \sigma_{t_1} \ldots \sigma_{t_N}, \quad N \geq 1,\\
	\text{$\omega$ is the empty word if $N = 0$}.
  \end{gathered}
\end{equation}
The word \eqref{ele.word} determines uniquely the set of spectral vectors \eqref{ele.spectrum} of the partition of $\mathbb R^2_+$ by the lines \eqref{ele.lines} and, as a corrollary, its convex conic hull $\eqref{ele.Gdef}$ for fixed $\rho \in [-1,0)\cup(0,+\infty)$.

Next, one varies the parameter $\rho$ of \eqref{ele.lines}. These variations change the partition of $\mathbb R^2_+$ by the lines \eqref{ele.lines} and, as a corollary, can transform the word \eqref{ele.word} associated to this partition. However, only two types of such transformations are possible:
\begin{subequations}
  \begin{align}
	& \sigma_{t_1} \sigma_{t_2} \leftrightarrow \sigma_{t_2} \sigma_{t_1}, \qquad\qquad\; \text{if $|t_1-t_2|\geq 2$}, \label{ele.transa} \\
	& \sigma_t \sigma_{t+1} \sigma_t \leftrightarrow \sigma_{t+1} \sigma_t \sigma_{t+1}, \quad (t=1,\dots,T-1).\label{ele.transb}
  \end{align}
\end{subequations}
Transformation \eqref{ele.transa} corresponds to a situation where the order, in which $R_\alpha$ meets the intersection points with the pairs of lines of \eqref{ele.lambda} as $\alpha$ changes from $0$ to $\tfrac \pi 2$, changes. Transformation \eqref{ele.transb} corresponds to situation where a line of \eqref{ele.lines} passes through the intersection point of two other lines of family \eqref{ele.lines}. 

Note that transformations of the form \eqref{ele.transa}, \eqref{ele.transb} appear in the braid theory, and are sometimes called the 2-braid and 3-braid moves, respectively. Besides, recall that the symmetric group  $S_T$ is generated by elements $\sigma_t$, \dots, $\sigma_{T-1}$, satisfying the Moore-Coxeter relations:
\begin{subequations}
\begin{align}
	& \sigma_t^2=1, \qquad\qquad\qquad\quad\; (t=1,\dots,T), \label{ele.MCa} \\
	& \sigma_{t_1}\sigma_{t_2}=\sigma_{t_2}\sigma_{t_1}, \qquad\qquad \text{if $|t_2-t_1|\geq 2$}, \label{ele.MCb}\\
	& \sigma_t\sigma_{t+1}\sigma_t=\sigma_{t+1}\sigma_t\sigma_{t+1}.\label{ele.MCc}
\end{align}
\end{subequations}
We show below that transformation \eqref{ele.transa} does not affect solvability of Problem \ref{ele.mompr} and thus one can factorise the set of formal words \eqref{ele.word} by relation \eqref{ele.MCb}. Hovewer, transformation \eqref{ele.transb} can affect solvability of Problem \ref{ele.mompr}. It turns out that it is convenient to study this effect using deformations of rhombic tilings.

\subsection{Rhombic tiling associated to partition and solvability of Problem \ref{ele.mompr}}\label{ele.til}

The cone $\Gamma_\rho( \widehat p )$ of \eqref{ele.Grhodef} does not change under variations of parameter $\rho$ corresponding to the word transformations \eqref{ele.transa}. However, it changes under variations of $\rho$ which correspond to the word transformations \eqref{ele.transb}. It is convenient to study the corresponding transformations of $\Gamma_\rho(\widehat p)$ using deformations of \textit{rhombic tilings}.

The \textit{rhombic tiling} corresponding to parameter $\rho \in [-1,0)\cup(0,+\infty)$ or, more precisely, to partition $\Lambda\{ \widetilde p(t,\rho) \mid t = 1,\dots, T\}$ of $\mathbb R^2_+$ by the lines \eqref{ele.lines}, is defined as follows. 

Fix $T$ different non-zero vectors $\xi_1$, \dots $\xi_T \in \mathbb R^2$. For definiteness, put
\begin{equation}\label{ele.xi_j}
  \xi_j = \bigl( j - \lfloor \tfrac T 2 \rfloor, 1 \bigr), \quad j = 1,\dots,T.
\end{equation}
To each domain $G$ of partition $\Lambda\{ \widetilde p(t,\rho) \mid t = 1,\dots, T\}$ we associate the vertex $\eta(G) \in \mathbb R^2$ of the rhombic tiling such that
\begin{equation}
  \eta(G) = Z_1(G) \xi_1 + \cdots + Z_T(G) \xi_T,
\end{equation}
where $Z(G) = (Z_1(G), \dots, Z_T(G))$ is the spectrum of $G$ defined in \eqref{ele.spectrum}.

Next, fix $\alpha \in (0,\tfrac \pi 2) \setminus \{\alpha_1,\dots,\alpha_N\}$ and let $R_\alpha$ be the ray of \eqref{ele.ray}. Let $G_{i_1}$, \dots, $G_{i_{T+1}}$ be the domains of partition $\Lambda\{ \widetilde p(t,\rho) \mid t = 1,\dots, T\}$ consecutively traversed by the point $\bigl(z_1,z_2(z_1)\bigr) \in R_\alpha$ as $z_1$ goes from $+\infty$ to $0$, and denote $V_1(\alpha) = \eta(G_{i_1})$, \dots, $V_{T+1}(\alpha) = \eta(G_{i_{T+1}})$. Note that
\begin{equation}\label{ele.etaGik}
  \begin{gathered}
  V_k(\alpha) = \xi_{\pi_k(\alpha)}+\xi_{\pi_{k+1}(\alpha)} + \cdots + \xi_{\pi_T(\alpha)} \quad (k=1,\dots,T), \quad 
  V_{T+1}(\beta) = 0, 
  \end{gathered}
\end{equation}
where $\pi(\alpha)$ is the permutation of \eqref{ele.pipermut}. The polygonal chain
\begin{equation}\label{ele.snake}
  Sn(\alpha) = \overline{V_1(\alpha)\cdots V_{T+1}(\alpha)}
\end{equation}
is called a \textit{snake} $Sn(\alpha)$ of the rhombic tiling.

Note that $Sn(\alpha) = Sn(\beta)$ all $\alpha$ , $\beta \in (\alpha_i,\alpha_{i+1})$ for fixed $i \in \{0,\dots,N\}$.

If $\alpha \in (\alpha_{i-1},\alpha_i)$, $\beta \in (\alpha_i,\alpha_{i+1})$, $i \in \{1,\dots,N\}$, then $Sn(\alpha)$ and $Sn(\beta)$ are different. More precisely:
\begin{equation}
  \begin{aligned}
	V_k(\beta) & = V_k(\alpha), \quad k \in \{1,\dots,T+1\}\setminus \{t_i+1\},\\
	V_{t_i+1}(\beta) & = V_{t_i+1}(\alpha) + \xi_{\pi_{t_i}(\alpha)} - \xi_{\pi_{t_i+1}(\alpha)}.
  \end{aligned}
\end{equation}

The rhombus with vertices $V_{t_i}(\alpha)$, $V_{t_i+1}(\alpha)$, $V_{t_i+2}(\alpha)$, $V_{t_i+1}(\beta)$ corresponds to elementary transposition $\sigma_{t_i}$.

The collection of all snakes $Sn(\alpha)$, $\alpha \in (0,\tfrac \pi 2) \setminus \{\alpha_1,\dots,\alpha_N\}$, and rhombuses corresponding to elementary transpositions is called a \textit{rhombic tiling}. 

Variations of $\rho$ corresponding to transformation \eqref{ele.transa} of word \eqref{ele.word} do not change the rombic tiling. Variations of $\rho$ corresponding to transformation \eqref{ele.transb} of word \eqref{ele.word} correspond to a so-called \text{flip} of the rhombic tiling. Roughly speaking, a flip of the rhombic tiling is the operation that changes the vertex $\xi_j  + \Xi_{-ijk}$ to the vertex $\xi_i + \xi_k + \Xi_{-ijk}$ (or vice versa), and makes the corresponding change of vertex in each snake containing this vertex. It is also required that the six common neighbors of these vertices given by
\begin{equation}
  \begin{gathered}
  \Xi_{-ijk}, \quad \xi_i + \Xi_{-ijk}, \quad \xi_k + \Xi_{-ijk}, \quad \xi_j+\xi_k + \Xi_{-ijk},\\ \xi_i+\xi_j+\Xi_{-ijk}, \quad \xi_i+\xi_j+\xi_k + \Xi_{-ijk}
  \end{gathered}
\end{equation}
belong to each of these rhombic tilings. Here $\xi_i$, $\xi_j$, $\xi_k$ are mutually different and $\Xi_{-ijk}$ is the sum of some of $\xi_s$, $s \not\in \{i,j,k\}$.

In what follows, making an appropriate renumbering, we assume that
\begin{equation}\label{ele.Sigmadef}
  \begin{gathered}
  \widetilde p_1(1,\rho) < \widetilde p_1(2,\rho) < \dots < \widetilde p_1(T,\rho),\\
  \widetilde p_2(\Sigma_\rho(1),\rho) < \widetilde p_2(\Sigma_\rho(2),\rho) < \dots < \widetilde p_2(\Sigma_\rho(T),\rho), \\
  \Sigma_\rho = (\Sigma_\rho(1),\dots,\Sigma_\rho(T)) = \Sigma \{ \widetilde p(t,\rho) \mid t = 1,\dots,T\} \in S_T,
  \end{gathered}
\end{equation}
for some fixed $\rho \in [-1,0) \cup (0,+\infty)$. Note that inequalities \eqref{ele.Sigmadef} remain valid if one varies $\rho$ in either of subintervals $[-1,0)$ or $(0,+\infty)$.

\begin{proposition}\label{ele.LZ} Let $\widetilde p_1(t,\rho)$, $\widetilde p_2(t,\rho)$, $\Sigma_\rho$ be defined by \eqref{ele.change<0}, \eqref{ele.Sigmadef} for $\rho<0$ and by \eqref{ele.change>0}, \eqref{ele.Sigmadef} for $\rho > 0$. Let $\omega_1 = \sigma_{i_1}\dots\sigma_{i_k}$ and $\omega_2 = \sigma_{j_1}\dots\sigma_{j_m}$ be the formal words of partitions corresponding to parameters $\rho_1$, $\rho_2 \in [-1,0) \cup (0,+\infty)$, respectively. If $\omega_1$ and $\omega_2$, considered as permutations, are equal to $\Sigma_\rho$, then they can be transformed one into another by a finite number of moves \eqref{ele.transa}, \eqref{ele.transb}.
\end{proposition}
Proposition \ref{ele.LZ} is a modification of a particular case of the results of \cite{Danilov2010}.

Proposition \ref{ele.LZ} is proved in Appendix \ref{ape.ele}.

Next, we give a sufficient and a necessary condition for solvability of Problem \ref{ele.mompr}.

\begin{proposition}\label{ele.snake+}
Let $\rho \in [-1,0) \cup (0,+\infty)$ be fixed. Let $\widetilde p_1(t,\rho)$, $\widetilde p_2(t,\rho)$ be defined by \eqref{ele.change<0}, \eqref{ele.Sigmadef} for $\rho<0$ and by \eqref{ele.change>0}, \eqref{ele.Sigmadef} for $\rho > 0$. Consider the permutation $\lambda=\bigl(\lambda(1),\dots,\lambda(T)\bigr) \in S_T$ such that
\begin{equation}\label{ele.lambda}
  y(\lambda(T)) < \cdots < y(\lambda(1)),
\end{equation}
where $y = (y(1),\dots,y(T))$ is the vector of outputs of \eqref{ele.mompr1}. We define the snake $Sn(\lambda)$ as the following polygonal chain:
\begin{equation}\label{ele.snlambda}
  \begin{gathered}
	\quad Sn(\lambda) = \overline{V_1(\lambda)\cdots V_{T+1}(\lambda)}, \\
	\quad V_k(\lambda) = \xi_{\lambda(k)} + \xi_{\lambda(k+1)} + \cdots + \xi_{\lambda(T)} \quad (k=1,\dots,T), \quad V_{T+1}(\lambda) = 0,
  \end{gathered}
\end{equation}
where $\xi_j$ are the vectors of \eqref{ele.xi_j}. If the rhombic tiling corresponding to partition $\Lambda\{ \widetilde p(t,\rho) \mid t = 1,\dots,T \}$ has the snake $Sn(\lambda)$, then $\rho$ is a solution of Problem \ref{ele.mompr}.
\end{proposition}

\begin{proposition}\label{ele.snake-}
Let $\rho \in [-1,0) \cup (0,+\infty)$ be fixed. Let $\widetilde p_1(t,\rho)$, $\widetilde p_2(t,\rho)$, $\Sigma_\rho$ be defined by \eqref{ele.change<0}, \eqref{ele.Sigmadef} for $\rho<0$ and by \eqref{ele.change>0}, \eqref{ele.Sigmadef} for $\rho > 0$. Consider the permutation $\lambda = \bigl(\lambda(1),\dots,\lambda(T)\bigr) \in S_T$ defined by \eqref{ele.lambda}, where $y = (y(1),\dots,y(T))$ is the vector of outputs of \eqref{ele.mompr1}. If $Sn(\lambda)$ does not belong to the closed bounded region bounded by $Sn(\id_{S_T})$ and $Sn(\Sigma_\rho)$, then $\rho$ is not a solution of Problem \ref{ele.mompr}. Here $Sn(\lambda)$, $Sn(\id_{S_T})$ and $Sn(\Sigma_\rho)$ are defined according to \eqref{ele.snlambda}
\end{proposition}
The proofs of Propositions \ref{ele.snake+} and \ref{ele.snake-} are given in Appendix \ref{ape.ele}.

Propositions \ref{ele.snake+}, \ref{ele.snake-} give a sufficient and a necessary condition for solvability of the moment problem \ref{ele.mompr}, respectively. These conditions are formulated in terms of snakes of rhombic tilings. If the snake $\lambda$ of \eqref{ele.lambda} corresponding to the time series of outputs $y = (y(1),\dots,y(T))$ is contained in the region bounded by the rhombic tiling, but does not belong to this tiling, solvability of moment problem is checked using Proposition \ref{ide.propsol}. It follows from Proposition \ref{ele.LZ} that applying a finite number of flips to this rhombic tiling, one can obtain a rhombic tiling that has this snake and share the same boundaries with the initial tiling (see the proof of Proposition \ref{ele.snake-} in Appendix). However, it is not always possible to get such a sequence of flips by varying parameter~$\rho$.

\appendix

\section{Proofs of the results of Section \ref{sec.agr}}\label{ape.agr}

\begin{proof}[Proof of Proposition \ref{agr.propsol}.] One can show that if a group of industries satisfies \eqref{agr.prod} and $l = (l_1,\dots,l_n)>0$, then optimization problem \eqref{agr.prmax}--\eqref{agr.prpos} satisfies Slater's condition (see \cite{Alekseev1987} for definition).

The Lagrange function $L = L(X^0,X^1,\dots,X^m,l^1,\dots,l^m,p_0,q,s)$ for convex programming problem \eqref{agr.prmax}--\eqref{agr.prpos} is given by 
\begin{equation}\label{agr.Ldef}
 \begin{aligned}
	L & = p_0 F_0(X^0) + \sum_{j=1}^m q_j \biggl( F_j(X^j,l^j)- \sum_{i=0}^m X^i_j \biggr) + \sum_{k=1}^n s_k \biggl( l_k - \sum_{j=1}^m l_k^j \biggr) \\
	& = s \cdot l + \bigl( p_0 F_0(X^0) - q \cdot X^0)  + \sum_{j=1}^m \bigl( q_j F_j(X^j,l^j) - q \cdot X^j - s \cdot l^j \bigr).
 \end{aligned}
\end{equation}
By the Kuhn-Tucker theorem (see, e.g. \cite{Alekseev1987}), the set of vectors $\{ \widehat X^0, \widehat X^1,\dots,\widehat X^m,\widehat l^1,\dots,\widehat l^m\}$ satisfying \eqref{agr.prbal}--\eqref{agr.prpos} is a solution of convex programming problem \eqref{agr.prmax}--\eqref{agr.prpos} with Slater condition if and only if:
\begin{enumerate}
 \item[(i)] There exist Lagrange multipliers $p_0 > 0$, $q = (q_1,\dots,q_m) \geq 0$, $s = (s_1,\dots,s_n) \geq 0$, such that the complementary slackness conditions \eqref{agr.qequ}, \eqref{agr.sequ} are satisfied.
 \item[(ii)] The maximum of Lagrange function $L$ on the set \eqref{agr.prpos} is attained at the set of vectors $\{ \widehat X^0, \widehat X^1,\dots,\widehat X^m,\widehat l^1,\dots,\widehat l^m\}$.
\end{enumerate}
It follows from formula \eqref{agr.Ldef} that the above condition (ii) is equivalent to \eqref{agr.argmaxcond}, \eqref{agr.argmaxcond2}. Proposition \ref{agr.propsol} is proved. 
\end{proof}

\begin{proof}[Proof of Proposition \ref{agr.propdual}.]

Using the Fenchel duality theorem (see, e.g., \cite{Aubin1984}) we have that
\begin{equation}\label{agr.PitoF2}
  F^A(l) = \tfrac{1}{p_0} \min\biggl\{ s \cdot l + \sum_{j=1}^m \Pi_j(q,s) \biggm| q_0(q) \geq p_0, \; q \geq 0, \; s \geq 0 \biggr\}, \quad p_0>0.
\end{equation}
Formulas \eqref{agr.Pidef}, \eqref{agr.PitoF2} imply \eqref{agr.PitoF}.

Using the Fenchel-Moreau theorem (see, e.g., \cite{Aubin1984}) and formula \eqref{agr.PitoF}, we obtain \eqref{agr.FtoPi}. Proposition \ref{agr.propdual} is proved. 
\end{proof}

\section{Proofs of the results of Section \ref{sec.ghj}}\label{ape.ghj}

\begin{proof}[Proof of Proposition \ref{ghj.propchar}] \textit{Necessity}. Let $\Pi(p,p_0)$ be defined by \eqref{ghj.phidef}. It follows from \eqref{ghj.fdef} that the definition is correct. Put $P_r(x_1,\dots,x_n) = (x_1^{-r},\dots,x_n^{-r})$. Then
\begin{equation}\label{ghj.dPidt}
  \begin{aligned}
  \tfrac{\partial \Pi}{\partial p_0}(p,p_0) & = \int_{\mathbb R^n_+} \theta\bigl( p_0 - h(p \circ x) \bigr) \varphi(x) \, dx  \\
   & = -r^{-1} \int_{\mathbb R^n_+} \int_0^{+\infty} \theta( p_0 - h(p \circ x)) t^{n-1}e^{-t} f(t P_r(x) ) dt d x_1^{-r} \dots d x_n^{-r} \\ 
   & = -r^{-1} \int_{\mathbb R^n_+}\int_0^{+\infty} \theta\bigl(p_0 - h( p \circ P_{1/r}(y) ) \bigr) t^{n-1} e^{-t} f(ty) \, dt dy \\
   & = -r^{-1} \int_{\mathbb R^n_+}\int_0^{+\infty} \theta\bigl( p_0^{-r} - P_r(p) \cdot y \bigr) t^{n-1} e^{-t} f(ty) \, dt dy \\
   & = -r^{-1} \int_{\mathbb R^n_+} \int_0^{+\infty} \theta(t - p_0^r P_r(p) \cdot w) \tfrac{e^{-t}}{t} f(w) dt dw \\
   & = -r^{-1} \int_{\mathbb R^n_+} f(w) \int_{p_0^r P_r(p) \cdot w}^{+\infty} \tfrac{e^{-t}}{t} dt \, dw.
  \end{aligned}
\end{equation}
Using \eqref{ghj.dPidt}, we obtain
\begin{equation}\label{ghj.d2Pidt}
  \tfrac{\partial^2 \Pi}{\partial p_0^2}(p,p_0) = p_0^{-1} \int_{\mathbb R^n_+} e^{-p_0^r P_r(p) \cdot w} f(w) \, dw.
\end{equation}
Formula \eqref{ghj.d2Pi} follows from \eqref{ghj.d2Pidt}. Formulas \eqref{ghj.Pilim}, \eqref{ghj.Pihomo} follow from \eqref{ghj.phidef}, \eqref{ghj.dPidt}.

\textit{Sufficiency.} Let $\widetilde \Pi(p,p_0)$ be a function satisfying \eqref{ghj.Pilim}--\eqref{ghj.d2Pi} and let $\Pi(p,p_0)$ be the function defined by \eqref{ghj.phidef}.

It follows from formulas \eqref{ghj.Pihomo}, \eqref{ghj.d2Pi} for $\widetilde \Pi$ and $\Pi$ that
\begin{equation}\label{ghj.d2Pi=d2tPi}
  \tfrac{\partial^2 \widetilde \Pi}{\partial p_0^2}(p,p_0) = \tfrac{\partial^2 \Pi}{\partial p_0^2}(p,p_0), \quad (p,p_0) > 0.
\end{equation}
Using formulas \eqref{ghj.Pihomo}, \eqref{ghj.d2Pi} for $\widetilde \Pi$ and $\Pi$ and formula \eqref{ghj.d2Pi=d2tPi}, we obtain that $\widetilde \Pi(p,p_0) = \Pi(p,p_0)$, $(p,p_0)>0$.
\end{proof}

\begin{proof}[Demonstration of Example \ref{ghj.exCD}] It is shown in \cite{Houth1955} that $F_\text{CD}(l_1,l_2)$ is the production function for the resource distribution problem \eqref{hjm.prmax}--\eqref{hjm.pru} with $\mu$ defined in \eqref{hjm.FCDconst}. Using formulas \eqref{hjm.Pidef}, \eqref{hjm.FCDconst}, we compute the profit function $\Pi_\text{CD}(p_1,p_2,p_0)$ corresponding to $F_\text{CD}(l_1,l_2)$:
\begin{equation}\label{ghj.PiCD}
  \Pi_\text{CD}(p_1,p_2,p_0) = A \tfrac{B(\alpha_1,\alpha_2)}{(\alpha_1+\alpha_2)(\alpha_1+\alpha_2+1)} p_0^{\alpha_1+\alpha_2+1} p_1^{-\alpha_1} p_2^{-\alpha_2}.
\end{equation}
Taking the second derivative of \eqref{ghj.PiCD} with respect to $p_0$, we get:
\begin{equation}
 \begin{gathered}
	\tfrac{\partial^2 \Pi_\text{CD}}{\partial p_0^2}(p_1,p_2,1) = A B(\alpha_1,\alpha_2) p_1^{-\alpha_1} p_2^{-\alpha_2}, \\
	\tfrac{\partial^2 \Pi_\text{CD}}{\partial p_0^2}(p_1^{-\frac 1 r},p_2^{-\frac 1 r},1) = AB(\alpha_1,\alpha_2) p_1^{\frac{\alpha_1}{r}} p_2^{\frac{\alpha_2}{r}}.
 \end{gathered}
\end{equation}
One can see that
\begin{equation}
 \begin{gathered}
	\tfrac{\partial^2 \Pi_\text{CD}}{\partial p_0^2}(p_1^{-\frac 1 r},p_2^{-\frac 1 r},1) = \int_{\mathbb R^2_+} e^{-p \cdot x} H_\text{CD}(x_1,x_2) dx_1 dx_2,\\
	H_\text{CD}(x_1,x_2) = \tfrac{AB(\alpha_1,\alpha_2)}{\Gamma(-\frac{\alpha_1}{r})\Gamma(-\frac{\alpha_2}{r})} x_1^{-\frac{\alpha_1}{r}-1} x_2^{-\frac{\alpha_2}{r}-1}.
 \end{gathered}
\end{equation}
Using Proposition \ref{ghj.propchar}, we obtain that
\begin{equation}\label{ghj.PiCDrepr}
  \begin{gathered}
  \Pi_\text{CD}(p_1,p_2,p_0) = \int_{\mathbb R^2_+} \bigl( p_0 - ( p_1^{-r} x_1^{-r} + p_2^{-r} x_2^{-r} )^{-\frac 1 r} \bigr)_+ \, \varphi_\text{CD}(x_1,x_2) \, dx_1 dx_2,\\ 
  \varphi_\text{CD}(x_1,x_2) = (-r)A \tfrac{B(\alpha_1,\alpha_2)}{B(-\frac{\alpha_1}{r},-\frac{\alpha_2}{r})} x_1^{\alpha_1-1} x_2^{\alpha_2-1}.
  \end{gathered}
\end{equation}
Formula \eqref{ghj.PiCDrepr} implies that $\Pi_\text{CD}(p_1,p_2,p_0)$ is the profit function for the resource distribution problem \eqref{ghj.prmax}--\eqref{ghj.pru} with \eqref{ghj.PiCDh}, \eqref{ghj.PiCDmu}. As a corollary, $F_\text{CD}(l_1,l_2)$ is the production function for the same problem.
\end{proof}

\begin{proof}[Demonstration of Example \ref{ghj.exCES}] The profit function $\Pi_\text{CES}(p,p_0)$ corresponding to $F_\text{CES}(l_1,l_2)$ is given by formula \eqref{hjm.PiCES}. We rewrite formula \eqref{hjm.PiCES} as follows:
\begin{gather}\label{ghj.PiCES}
	\Pi_\text{CES}(p_1,p_2,p_0) = \gamma^{\frac{1}{1-\gamma}}(1-\gamma) p_0^{\frac{1}{1-\gamma}} \bigl( \beta_1 p_1^{-\frac r 2} + \beta_2 p_2^{-\frac r 2} \bigr)^{-b}, \\
	\beta_1 = \alpha_1^{\frac{1}{1+\rho}}, \; \beta_2 = \alpha_2^{\frac{1}{1+\rho}}, \; r = -\tfrac{2\rho}{1+\rho}, \; b = \tfrac{\gamma}{1-\gamma} \tfrac{1+\rho}{\rho}. \label{ghj.PiCESconsts}
\end{gather}
Taking the second derivative of \eqref{ghj.PiCES} with respect to $p_0$, we get:
\begin{equation}
	\begin{gathered}
   \tfrac{\partial^2 \Pi_\text{CES}}{\partial p_0^2}(p_1,p_2,1) = \tfrac{\gamma^{\frac{1}{1-\gamma}}}{1-\gamma} \bigl( \beta_1 p_1^{-\frac r 2} + \beta_2 p_2^{-\frac r 2} \bigr)^{-b}, \\
   \tfrac{\partial^2 \Pi_\text{CES}}{\partial p_0^2}(p_1^{-\frac 1 r},p_2^{-\frac 1 r},1) = \tfrac{\gamma^{\frac{1}{1-\gamma}}}{1-\gamma} \bigl( \beta_1 \sqrt{p_1} + \beta_2 \sqrt{p_2} \bigr)^{-b}.
   \end{gathered}
 \end{equation}
Using Lemma \ref{app.lemmaLapsqrt2} formulated below we have that:
\begin{equation}
 \begin{gathered}
	\tfrac{\partial^2 \Pi_\text{CES}}{\partial p_0^2}(p_1^{-\frac 1 r},p_2^{-\frac 1 r},1) = \int_{\mathbb R^2_+} e^{-p \cdot x} H_\text{CES}(x_1,x_2) \, dx_1 dx_2,\\
	H_\text{CES}(x_1,x_2) = \tfrac{\gamma^{\frac{1}{1-\gamma}}}{1-\gamma}  \tfrac{2^{b-1}}{\pi} \beta_1 \beta_2 \tfrac{\Gamma(\frac b 2 +1)}{\Gamma(b)} (x_1 x_2)^{-\frac 3 2} \bigl( \tfrac{\beta_1^2}{x_1} + \tfrac{\beta_2^2}{x_2} \bigr)^{-\frac b 2 - 1}.
 \end{gathered}
\end{equation}
Using Proposition \ref{ghj.propchar}, we obtain that
\begin{equation}\label{ghj.PiCESrepr}
\begin{gathered}
  \Pi_\text{CES}(p_1,p_2,p_0) = \int_{\mathbb R^n_+} \bigl( p_0 - ( p_1^{-r} x_1^{-r} + p_2^{-r} x_2^{-r} )^{-\frac 1 r} \bigr)_+ \varphi_\text{CES}(x_1,x_2) \, dx_1 dx_2, \\
  \varphi_\text{CES}(x_1,x_2) = \tfrac{\gamma^{\frac{1}{1-\gamma}}}{1-\gamma}  \tfrac{2^{b-1}}{\pi} \beta_1 \beta_2 \tfrac{\Gamma(\frac b 2 +1)\Gamma(\frac b 2)}{\Gamma(b)} (x_1 x_2)^{\frac r 2-1} \bigl( \beta_1^2 x_1^r + \beta_2^2 x_2^r \bigr)^{-\frac b 2 - 1}.
  \end{gathered}
\end{equation}
Formula \eqref{ghj.PiCESrepr} and definitions \eqref{ghj.PiCESconsts} imply that $\Pi_\text{CES}(p_1,p_2,p_0)$ is the profit function for the resource distribution problem \eqref{ghj.prmax}--\eqref{ghj.pru} with \eqref{ghj.PiCESh}, \eqref{ghj.PiCESmu}. As a corollary, $F_\text{CES}(l_1,l_2)$ is the production function for the same problem.
\end{proof}

\begin{lemma}\label{app.lemmaLapsqrt} Let $f(s)$ be a function such that $\int_0^{+\infty} e^{-As} |f(s)| \, ds < \infty$ for any $A>0$.
Then
\begin{equation}\label{app.Lapsqrt}
 \begin{gathered}
	\int_0^{+\infty} e^{-s(\sqrt p_1 + \sqrt p_2)} f(s) \, ds = \int_{\mathbb R^2_+} e^{-p \cdot x} H(x) \, dx, \quad p = (p_1,p_2)>0, \\
	H(x_1,x_2) = \tfrac{1}{8\pi} (x_1 x_2)^{-\frac 3 2} G( \tfrac{1}{4x_1} + \tfrac{1}{4x_2} ),\\
	G(s) = \int_0^{+\infty} e^{-st} \sqrt t f(\sqrt t) \, dt.
 \end{gathered}
\end{equation}
\end{lemma}
\begin{proof} The following formula is well-known, see, e.g., \cite{Bateman1954}:
\begin{equation*}
  \frac{1}{4\pi} \int_{\mathbb R^2_+} e^{-p\cdot x} \frac{s_1 s_2}{(x_1 x_2)^{\frac 3 2}} e^{-\frac{s_1^2}{4x_1}-\frac{s_2^2}{4x_2}} dx_1 dx_2 = e^{-s_1 \sqrt{p_1} - s_2 \sqrt{p_2}},
\end{equation*}
where $s = (s_1,s_2) > 0$, $p = (p_1,p_2) > 0$. We set $s_1 = s_2$, multiply this equation by $f(s)$ and integrate over $s \in [0,+\infty)$:
\begin{equation*}
  \begin{gathered}
  \frac{1}{4\pi} \int \frac{e^{-p \cdot x}}{(x_1 x_2)^{\frac 3 2}} \int_0^{+\infty} s^2 e^{-s^2\bigl( \frac{1}{4x_1}+\frac{1}{4x_2} \bigr)} f(s) ds \, dx_1 dx_2  \\
  = \int_0^{+\infty} e^{-s(\sqrt{p_1}+\sqrt{p_2})} f(s) \, ds.
  \end{gathered}
\end{equation*}
Making the change of variable $s^2 = t$ in the inner integral on the left, we get \eqref{app.Lapsqrt}. 
\end{proof}

\begin{lemma}\label{app.lemmaLapsqrt2} Let $\beta_1$, $\beta_2$, $b > 0$. The following formula is valid:
\begin{equation}\label{app.Lapsqrt2}
  \begin{gathered}
	( \beta_1 \sqrt{p_1} + \beta_2 \sqrt{p_2})^{-b} = \int_{\mathbb R^2_+} e^{-p\cdot x} \widetilde H(x_1,x_2) dx_1 dx_2, \quad p = (p_1,p_2)>0,\\
	H(x_1,x_2) = \tfrac{2^{b-1}}{\pi} \beta_1 \beta_2 \tfrac{\Gamma(\frac b 2 +1)}{\Gamma(b)} (x_1 x_2)^{-\frac 3 2} \bigl( \tfrac{\beta_1^2}{x_1} + \tfrac{\beta_2^2}{x_2} \bigr)^{-\frac b 2 - 1}, \quad (x_1,x_2) > 0,
  \end{gathered}
\end{equation}
where $\Gamma$ is the gamma function.
\end{lemma}

\begin{proof} Consider the case $\beta_1 = \beta_2 = 1$. Put $f(s) = \frac{s^{b-1}}{\Gamma(b)}$. We have
\begin{equation}\label{app.Lf}
  \int_0^{+\infty} e^{-s(\sqrt{p_1}+\sqrt{p_2})} f(s) \, ds = (\sqrt{p_1}+\sqrt{p_2} )^{-b}.
\end{equation}
We define functions $G(s)$ and $H(x_1,x_2)$ according to \eqref{app.Lapsqrt}. Then
\begin{equation}\label{app.GHsqrt}
 \begin{gathered}
  G(s) = \int_0^{+\infty} e^{-st} \frac{t^{\frac b 2}}{\Gamma(b)} dt =  s^{-\frac b 2 - 1} \tfrac{\Gamma(\frac b 2 + 1)}{\Gamma(b)}, \\
  H(x_1,x_2) = \tfrac{2^{b-1}}{\pi} \tfrac{\Gamma(\frac b 2 + 1)}{\Gamma(b)} (x_1 x_2)^{-\frac 3 2} \bigl( \tfrac{1}{x_1} + \tfrac{1}{x_2} \bigr)^{-\frac b 2 - 1}.
 \end{gathered}
\end{equation}
Using \eqref{app.Lapsqrt}, \eqref{app.Lf}, \eqref{app.GHsqrt}, we get \eqref{app.Lapsqrt2} with $\widetilde H = H$ for $\beta_1 = \beta_2 = 1$.

Besides, recall the scaling property of the Laplace transform:
\begin{equation*}
  \int_{\mathbb R^2_+} e^{-\beta_1^2 p_1 x_1 - \beta_2^2 p_2 x_2} H(x) dx = \frac{1}{\beta_1^2 \beta_2^2} \int_{\mathbb R^2_+} e^{-p\cdot x} H(\tfrac{x_1}{\beta_1^2},\tfrac{x_2}{\beta_2^2})\,dx.
\end{equation*}
It follows from this scaling property that for arbitrary $\beta_1$, $\beta_2$ we have \eqref{app.Lapsqrt2} with $\widetilde H(x_1,x_2) = (\beta_1 \beta_2)^{-2} H(\tfrac{x_1}{\beta_1^2},\tfrac{x_2}{\beta_2^2})$.

Lemma \ref{app.lemmaLapsqrt2} is proved.
\end{proof}

\section{Proofs of the results of Section \ref{sec.ide}}\label{ape.ide}

\begin{proof}[Proof of Proposition \ref{ide.propnes}.] It follows from \eqref{ide.wdef} that
\begin{equation}\label{ide.NiGdef}
\begin{gathered}
  \sum_{t \in \Omega_i} w_G(t) = N_i(G), \quad G \in \Lambda_h ( \widehat p ) \quad (i =1,2),\\
  N_i(G) = \bigl| \bigl\{ t \in \Omega_i \mid G \subseteq \{ x \mid p_0(t) > h( p(t) \circ x ) \} \bigr\} \bigr|,
\end{gathered}
\end{equation}
where $|S|$ denotes the number of elements of set $S$.

Suppose that Problem \ref{ide.mompr} is solvable and $\mu$ is a solution. Then using \eqref{ide.mompreq} we get
\begin{equation}\label{ide.ysum}
  \sum_{t \in \Omega_i} y(t) = \sum_{G \in \Lambda_h ( \widehat p )} \mu(G) N_i(G) \quad (i = 1,2),\\
\end{equation}
where $N_i(G)$ is defined in \eqref{ide.NiGdef}.

Formulas \eqref{ide.w1>w2}, \eqref{ide.NiGdef}, \eqref{ide.ysum} and non-negativity of measure $\mu$ imply \eqref{ide.y1>y2}.
\end{proof}

\begin{proof}[Proof of Proposition \ref{ide.propintnorm}.] Consider an arbitrary face of cone $\Gamma_h(\widehat p) = \Gamma_h\{ \widehat p(t) \mid t = 1,\dots, T\}$. This face is the linear span of some linearly independent spectra $Z(G_1)$, \dots, $Z(G_{T-1})$ of $G_1$, \dots, $G_{T-1} \in \Lambda_h (  \widehat p )$. Thus, it can be described by the following equation for $Z = (Z_1,\dots,Z_T)$:
\begin{equation}\label{ide.faceeq}
  \det \begin{pmatrix} Z \\ Z(G_1) \\ \dots \\ Z(G_{T-1}) \end{pmatrix} = 0.
\end{equation}
As far as the coordinates of $Z(G_1)$, \dots, $Z(G_{T-1})$ are integer, the coefficients of equation \eqref{ide.faceeq} are also integer. Besides, this vector of coefficients is a normal vector to the face described by \eqref{ide.faceeq}. 
\end{proof}

\begin{proof}[Proof of Proposition \ref{ide.propsuf}. Sufficiency.] Suppose that $\Gamma_h(\widehat p) = \Gamma_h\{ \widehat p(t) \mid t = 1,\dots,T \}$ is discretely convex. Let $y = ( y(1), \dots, y(T) )$ be a vector satisfying the necessary condition of Proposition \ref{ide.propnes}. In view of Proposition \ref{ide.propsol}, it is sufficient to show that $y \in \Gamma_h (\widehat p )$.

In turn, in order to show that $y \in \Gamma_h ( \widehat p )$, it is sufficient to check that for any inner normal $\nu = (\nu_1,\dots,\nu_T)$ to $\Gamma_h(\widehat p)$, whose coordinates belong to $\{-1,0,1\}$, we have $y \cdot \nu \geq 0$.

Put
\begin{equation}\label{ide.Om12def}
  \Omega_1 = \{ j \mid \nu_j = 1\}, \quad \Omega_2 = \{j \mid \nu_j = -1 \}.
\end{equation}
Using that $z(G) \in \Gamma_h(\widehat p)$, $G \in \Lambda_h(\widehat p)$, we have that
\begin{equation}\label{ide.znu>0}
  \sum_{t \in \Omega_1} w_G(t) - \sum_{t \in \Omega_2} w_G(t) = Z(G) \cdot \nu \geq 0,  \quad G \in \Lambda_h ( \widehat p ).
\end{equation}

By hypothesis, $y$ satisfies the necessary condition of Proposition \ref{ide.propnes}. Thus, \eqref{ide.znu>0} implies
\begin{equation}
  y \cdot \nu = \sum_{t \in \Omega_1} y(t) - \sum_{t \in \Omega_2} y(t) \geq 0.
\end{equation}
Hence, $y \in \Gamma_h ( \widehat p )$.

\textit{Necessity.} For a given pair $(\Omega_1,\Omega_2)$ of subsets $\Omega_1$, $\Omega_2 \subset \{1,\dots,T\}$ such that $\Omega_1 \cap \Omega_ 2 = \varnothing$, $\Omega_1 \cup \Omega_2 \neq \varnothing$, we set
\begin{equation}\label{ide.nudef}
  \begin{gathered}
  \nu_{\Omega_1,\Omega_2} =  \bigl( \nu_{\Omega_1,\Omega_2}(1), \dots, \nu_{\Omega_1,\Omega_2}(T) \bigr),\\
  \nu_{\Omega_1,\Omega_2}(t) = \begin{cases}
								  1, & t \in \Omega_1, \\
								  -1, & t \in \Omega_2, \\
								  0, & t \not\in \Omega_1 \cup \Omega_2.
                               \end{cases}
  \end{gathered}
\end{equation}
Let
\begin{equation}\label{ide.Gnudef}
  \begin{gathered}
	M = \{ (\Omega_1,\Omega_2) \mid \Omega_1 \cap \Omega_2 = \varnothing, \Omega_1 \cup \Omega_2 \neq \varnothing, \; \text{$\eqref{ide.w1>w2}$ holds} \}, \\
	N = \bigl\{ \nu_{\Omega_1,\Omega_2} \mid (\Omega_1,\Omega_2) \in M \bigr\},\\
	\Gamma = \{ x \mid \forall \nu \in N \;\; x \cdot \nu \geq 0 \}.
  \end{gathered}
\end{equation}
Using these notations, Proposition \ref{ide.propnes} can be reformulated as $\Gamma_h(\widehat p) \subseteq \Gamma$. Suppose that the necessary condition of Proposition \ref{ide.propnes} is also sufficient. Then
\begin{equation}\label{ide.GeqGnu}
  \Gamma_h (\widehat p) = \Gamma.
\end{equation}
By construction, the cone $\Gamma_h(\widehat p)$ is spanned by vectors with coordinates in $\{-1,0,1\}$.  It also follows from \eqref{ide.nudef}, \eqref{ide.Gnudef}, \eqref{ide.GeqGnu} that each face of $\Gamma_h ( \widehat p )$ admits a non-zero normal vector with coordinates in $\{-1,0,1\}$. It follows that $\Gamma_h ( \widehat p )$ is discretely convex.
\end{proof}

\section{Proofs of the results of Section \ref{sec.ele}}\label{ape.ele}

\begin{proof}[Proof of Proposition \ref{ele.LZ}.] We prove that using moves \eqref{ele.transa}, \eqref{ele.transb} one can tranform $\omega_1$ and $\omega_2$ to a fixed word depending only on $\Sigma_\rho = \Sigma\{ \tilde p(t,\rho) \mid t = 1,\dots,T \} \in S_T$. One can see that it is true for $T=2$. For arbitrary $T \geq 3$ we prove it using induction.

Let $t^* \in \{1,\dots,T\}$ be such that
\begin{equation}
 \widetilde p_2(t^*,\rho)=\max_{t=1,\dots,T}\left\{\widetilde p_2(t,\rho)\right\}.
\end{equation}	

We continuously increase $\widetilde p_2(t^*,\rho)$ leaving the parameters $\widetilde p_1(t,\rho)$, $t = 1$, \dots, $T$, and $\widetilde p_2(t,\rho)$, $t = 1$, \dots, $t^*-1$, $t^*+1$, \dots, $T$, unchanged. This leads to transformations of the word of \eqref{ele.word} associated to the partition of $\mathbb R^2_+$ by the lines of \eqref{ele.lines}. 

As we increase $\alpha$ from $0$ to $\tfrac \pi 2$, the ray $R_\alpha$ consecutively meets intersection points \eqref{ele.alphas} of the pairs of lines of \eqref{ele.lines}. Variations of $\widetilde p_2(t^*,\rho)$ can change the order in which $R_\alpha$ meets these points. This situation corresponds to application of move \eqref{ele.transa} to the word \eqref{ele.word} associated to partition.  

Besides, if one increases $\widetilde p_2(t^*,\rho)$, the $t^*$-th line of \eqref{ele.lines} can meet the intersection point for a pair of other lines of \eqref{ele.lines}. This situation corresponds to application of move \eqref{ele.transb} to the word \eqref{ele.word} associated to partition.

One can see that the $t^*$-th line of \eqref{ele.lines} intersects the lines with numbers $t = t^*+1$, \dots, $T$ and does not intersect the other lines of \eqref{ele.lines}. Let $\alpha_{tt^*}$, $t = t^*+1$, \dots, $T$, be the values of $\alpha$ for which $R_\alpha$ meets the intersection point of the $t^*$-th line with the $t$-th line. Note that
\begin{equation}
  \alpha_{tt^*} \to 0, \quad \text{as $p_2(t^*,\rho)\to+\infty$}, \quad T = t^*+1,\dots,T.
\end{equation}
Choose $p_2(t^*,\rho)$ in such a way that $\alpha_{tt^*} < \alpha^*$, $t=t^*+1$, \dots, $T$, where $\alpha^*$ be the minimal of the angles for which the ray $R_\alpha$ meets an intersection point of a pair of lines of \eqref{ele.lines} with $t \neq t^*$. 

As $\alpha$ increases from $0$ to $\tfrac \pi 2$, the ray $R_\alpha$ consecutively meets the intersection points of the pairs of lines of \eqref{ele.lines} with numbers $(t^*,t^*+1)$, \dots, $(t^*,T)$ and then the intersection points of the lines of \eqref{ele.lines} different from the $t^*$-th line. As a corollary, the word \eqref{ele.word} corresponding to this partition of $\mathbb R^2_+$ by the lines of \eqref{ele.lines} starts as
\begin{equation}\label{ele.wordbeg}
 \sigma_{t^*}\sigma_{t^*+1}\dots\sigma_{T-2}\sigma_{T-1}
\end{equation}
Note that this subword is completely determined by $\Sigma_\rho = \Sigma\{ \widetilde p(t,\rho) \mid t = 1,\dots,T\} \in S_T$.

Also note that if $\alpha>\alpha^*$ then $\pi_T(\alpha)=t^*$ and the $t^*$-th line does not appear in intersections of $R_\alpha$ with the pairs of lines of \eqref{ele.lines}. Thus, symbol $\sigma_{T-1}$ does not appear in the next part of the word \eqref{ele.word}.

Then remove the $t^*$-th line from family \eqref{ele.lines}, renumerate remaining lines and define the new permutation $\Sigma' \in S_{T-1}$ according to \eqref{ele.Sigmadef}. The new permutation $\Sigma'$ is uniquely determined by $\Sigma_\rho$.

The word \eqref{ele.word} corresponding to the new partition is obtained from the old word by removing the beginning \eqref{ele.wordbeg}.

It remains to apply the induction hypothesis to the new partition and corresponding formal word \eqref{ele.word}.

Proposition \ref{ele.LZ} is proved.
\end{proof}

\begin{proof}[Proof of Proposition \ref{ele.snake+}.]
Let $\alpha \in (0,\tfrac \pi 2)$ be such that $\pi(\alpha)=\lambda$, where $\pi(\alpha)$ is defined in \eqref{ele.pipermut}. Let $G_0,\dots,G_T$ be the domains of partition $\Lambda\{ \widetilde p(t,\rho) \mid t = 1,\dots,T \}$ consecutively traversed by the point $(z_1,z_2(z_1)) \in R_\alpha$ as $z_1$ goes from $+\infty$ to $0$. 

Let $\zeta_j \in \mathbb R^2_+$, $r_j > 0$, $j = 1$, \dots, $T$, be such that
\begin{equation}
  B_{r_j}(\zeta_j)= \bigl\{ x\in \mathbb R^2 \mid | x - \zeta_j | < r_j \bigr\} \subset G_j \quad (j=1,\dots,T). 
\end{equation}
Put
\begin{equation}
\begin{gathered}
  f(x) = \frac{y(\lambda(T))}{\pi r_T^2}, \quad x \in B_{r_T}(\zeta_T), \\
  f(x) = \frac{ y(\lambda(T-j))-y(\lambda(T-j+1))}{\pi r_{T-j}^2}, \quad x\in B_{r_j}(\zeta_j), \; 1 \leq j \leq T-1, \\
  f(x) = 0, \quad x \not\in B_{r_j}(\zeta_j), \; 1 \leq j \leq T.
\end{gathered}
\end{equation}
Then $\mu(dx) = f(x) dx$ is a non-negative absolutely continuous measure which solves the moment problem \eqref{ele.mompr1}.
\end{proof}

\begin{proof}[Proof of Proposition \ref{ele.snake-}.] We begin by showing that if $Sn(\lambda)$ does not belong to the closed bounded domain bounded by $Sn(\id_{S_T})$ and $Sn(\Sigma_\rho)$, then 
\begin{equation}\label{ele.existt1t2}
  \begin{gathered}
 \text{there exist $t_1$, $t_2 \in \{1,\dots,T\}$ such that}\\
  t_1 < t_2, \quad \widetilde p_2(t_1,\rho) < \widetilde p_2(t_2,\rho), \quad y(t_1)<y(t_2).
  \end{gathered}
\end{equation}
Suppose that \eqref{ele.existt1t2} is not true. Let $Y_t = (y(t),y(t))$, $t = 1$, \dots, $T$. We join $Y_t$ by the line segments with $(\tfrac{1}{\widetilde p_1(t,\rho)},0)$ and $(0,\tfrac{1}{\widetilde p_2(t,\rho)})$. We call this pair of segments a \textit{wire}.

If \eqref{ele.existt1t2} does not hold, each pair of wires has at most one intersection point, and thus the set of these wires is a strict wiring diagram. The boundaries of the closed domain bounded by this diagram coincide with boundaries of the closed domain bounded by the rhombic tiling corresponding to partition $\Lambda\{ \widetilde p(t,\rho) \mid t =  1,\dots,T \}$. This closed domain contains $Sn(\lambda)$, and this contradicts the hypothesis.

Next, suppose that $\mu$ is a solution of the moment problem \eqref{ele.mompr1}, and suppose that $Sn(\lambda)$ is not contained in the closed bounded region bounded by $Sn(\id_{S_T})$ and $Sn(\Sigma_\rho)$. As it was shown above, we have \eqref{ele.existt1t2}.

Condition $\widetilde p_2(t_1,\rho) < \widetilde p_2(t_2,\rho)$ implies that
\begin{equation}\label{ele.pincls}
  \begin{gathered}
  \bigl\{ (z_1,z_2)\in\mathbb R^2_+ \mid \widetilde p_1(t_2,\rho)z_1 + \widetilde p_2(t_2,\rho)z_1 < 1 \bigr\} \qquad\qquad\qquad \\
	\qquad\qquad\qquad\subset \bigl\{ (z_1,z_2) \in \mathbb R^2_+ \mid \widetilde p_1(t_1,\rho)z_1 + \widetilde p_2(t_1,\rho)z_2 < 1 \}.
  \end{gathered}
\end{equation}
Applying Proposition \ref{ide.propnes} with $\Omega_1 = \{t_1\}$, $\Omega_2 = \{t_2\}$ and using \eqref{ele.pincls}, we have that $y(t_2) \leq y(t_1)$. It contradicts \eqref{ele.existt1t2}. Thus, the moment problem \ref{ele.mompr1} is not solvable.

Proposition \ref{ele.snake-} is proved.
\end{proof}

 \section*{Aknowledgements}
 The present work is supported by the Russian Science Foundation (project number 16-11-10246).

\bibliographystyle{IEEEtranS}
\bibliography{biblio_utf}    

\end{document}